\documentclass[11pt]{amsart}
\usepackage{amsmath, amsthm, amssymb, fullpage,amscd}
\usepackage[numbers]{natbib}
\usepackage{hyperref}
\usepackage{graphicx}
\usepackage{url}
\usepackage{subcaption}
\usepackage[dvipsnames]{xcolor}
\definecolor{blue}{rgb}{0,0,1}

\usepackage{tikz}
\usetikzlibrary{calc,arrows}
\usetikzlibrary{decorations.pathreplacing}
\usetikzlibrary{shapes,arrows}

\newcommand{\QQ}{\mathbb{Q}}

\newcommand\be{\begin{equation}}
\newcommand\ee{\end{equation}}
\newcommand\bea{\begin{eqnarray}}
\newcommand\eea{\end{eqnarray}}
\newcommand\bi{\begin{itemize}}
\newcommand\ei{\end{itemize}}
\newcommand\ben{\begin{enumerate}}
\newcommand\een{\end{enumerate}}
\newcommand\bc{\begin{center}}
\newcommand\ec{\end{center}}
\newcommand\ba{\begin{array}}
\newcommand\ea{\end{array}}

\newcommand{\Z}{\ensuremath{\mathbb{Z}}}
\newcommand{\Q}{\mathbb{Q}}

\theoremstyle{definition}
\newtheorem{thm}{Theorem}[section]
\newtheorem{conj}[thm]{Conjecture}

\newtheorem{lem}[thm]{Lemma}
\newtheorem{prop}[thm]{Proposition}

\newtheorem{defn}[thm]{Definition}

\newtheorem{ex}[thm]{Example}
\newtheorem{que}[thm]{Question}

\newtheorem{proj}[thm]{Research Project}
\newtheorem{obs}[thm]{Observation}

\newcommand{\orien}[1]{\mathcal{#1}}

\numberwithin{equation}{section}

\newcommand{\lc}{\left\lceil}
\newcommand{\rc}{\right\rceil}
\newcommand{\ceil}[1]{\lc #1 \rc}
\newcommand{\lf}{\left\lfloor}
\newcommand{\rf}{\right\rfloor}
\newcommand{\floor}[1]{\lf #1 \rf}

\newcommand{\dbdi}[2]{DBDI_{#1}(#2)}

\newcommand{\interval}[1]{DBDI_{#1}}

\title{On $(t,r)$ broadcast domination of directed graphs}

\author{Pamela E. Harris}
\author{Peter Hollander}
\author{Erik Insko}
\address[P.~E.~Harris, P.~Hollander]{Williams College,
Department of Mathematics and Statistics,
Williamstown, MA, USA}

\email[P.~E.~Harris]{\textcolor{blue}{\href{mailto:peh2@williams.edu}{peh2@williams.edu}}}

\email[P.~Hollander]{\textcolor{blue}{\href{mailto:pjh1@williams.edu}{pjh1@williams.edu}}}

\address[E.~Insko]{Florida Gulf Coast University,
Department of Mathematics,
 Fort Myers, FL, USA}

\email[E.~Insko]{\textcolor{blue}{\href{mailto:einsko@fgcu.edu}{einsko@fgcu.edu}}}

\begin{document}

\begin{abstract}
    A \emph{dominating set} of a graph $G$ is a set of vertices that contains at least one endpoint of every edge on the graph. The \emph{domination number} of $G$ is the order of a minimum dominating set of $G$. The \emph{$(t,r)$ broadcast domination} is a generalization of domination in which a set of broadcasting vertices emits signals of strength $t$ that decrease by 1 as they traverse each edge, and we require that every vertex in the graph receives a cumulative signal of at least $r$ from its set of broadcasting neighbors. In this paper, we extend the study of $(t,r)$ broadcast domination to directed graphs. Our main result explores the interval of values obtained by considering the directed $(t,r)$ broadcast domination numbers of all orientations of a graph $G$. In particular, we prove that in the cases $r=1$ and $(t,r) = (2,2)$, for every integer value in this interval, there exists an orientation $\vec{G}$ of $G$ which has directed $(t,r)$ broadcast domination number equal to that value. We also investigate directed $(t,r)$ broadcast domination on the finite grid graph, the star graph, the infinite grid graph, and the infinite triangular lattice graph. We conclude with some directions for future study.
\end{abstract}


\maketitle

\section{Introduction}\label{section}
Throughout this work we let $G$ denote a finite simply connected graph, and let $V(G)$ and $E(G)$ denote its vertex and edge set, respectively. A \emph{dominating set} of a graph $G$ is a subset $S \subseteq V(G)$ of vertices such that every vertex of $G$ is either in $S$ or is adjacent to a vertex in $S$ by an edge in $E(G)$. The \emph{domination number} of $G$, denoted $\gamma(G)$, is the minimal cardinality of a dominating set of $G$. That is, \[\gamma(G)=\min\{|S|\,:\;\mbox{$S$ is a dominating set of $G$}\}.\]

The concept of graph domination was first introduced by Claude Berge in the 1950's and 1960's~\cite{introIdeasToGraphDom}, and the terms \emph{dominating set} and \emph{domination number} were first formally defined by Oystein Ore in 1962 \cite{firstDefOfDomNum}. Thousands of papers have since been published on the subject, and many long-standing conjectures have been made.
For a comprehensive volume on developments and results in this very active area of research, we recommend the texts by Haynes, Hedetniemi and Slater \cite{haynes2017domination,haynes1998fundamentals}. In these texts, the authors also summarize many interesting variants of graph domination and provide a plethora of open problems.

Among the many variants of the graph domination problem, one important variant is \emph{distance domination}. Distance domination allows a vertex in the dominating set to ``dominate'' not only the vertices directly adjacent to it, but it also allows for that vertex to ``dominate'' all vertices within a certain distance of it.
More precisely, a \emph{distance-$k$ dominating set} of a graph $G$ is a set $S \subseteq V(G)$ of vertices such that every vertex $v\in V(G)$ is either in $S$ or there exists a vertex $s\in S$ such that the distance between $v$ and $s$ is at most $k$. 
In other words, $S$ is a distance-$k$ dominating set of $G$ if and only if we can reach any vertex $w$ in $G$ by following a path of length at most $k$ which starts on a vertex in $S$ and terminates at $w$. Note that the length of a path is just the number of edges in that path. Similar to the standard domination number of a graph, the \emph{distance-$k$ domination number} of a graph $G$ is the minimal cardinality of a distance-$k$ dominating set of $G$. We remark that distance-$k$ domination is a generalization of standard domination, as the definition of standard domination is equivalent to that of distance-1 domination.

The concept of distance domination was first motivated in 1991 by Henning, Oellermann, and Swart \cite{distanceDomFirstPaper}, and many problems related to this graph parameter which are stated in the previously mentioned works of Haynes, Hedetniemi, and Slater \cite{haynes2017domination,haynes1998fundamentals} remain open and motivate further study.
Some recent works on distance-$k$ domination include work on finding upper bounds for the distance domination numbers of finite grid graphs by Farina and Grez \cite{FarinaGrez2016}, as well as work by Drews, Harris, and Randolph on the density of distance dominating sets for the infinite grid graph in \cite{drews2017optimal}.

For positive integer parameters $t$ and $r$, a further generalization of a graph's domination number is known as the
$(t,r)$ broadcast domination number, first defined by Blessing, Insko, Johnson, and Mauretour in 2014 \cite{blessing2014tr}. 
This is a two-parameter family of graph invariants which generalizes standard domination and distance domination. The concept of broadcast domination can first be understood with the following analogy: consider a set of broadcast towers placed on a subset of the vertices of a graph, each with a known signal strength $t$ (the easiest graph to picture is a grid, but the process works in the same way on all graphs). Each tower gives itself signal strength $t$, each neighbor of this tower receives signal strength $t-1$, each neighbor’s neighbor receives signal strength $t-2$, and so on, until the signal from a tower dies out (i.e., reaches zero). If there are multiple towers whose signal reaches a single cell phone (a vertex of $G$ that is not a tower), those signal strengths are added together. Then, given a positive integer $r$, the $(t,r)$ broadcast domination number of $G$ is the minimal number of towers of signal strength $t$ needed to ensure that every vertex of the graph receives signal strength at least $r$. 

The $(t,r)$ broadcast domination parameter, first studied by Blessing et. al.\ in~\cite{blessing2014tr}, has since been studied extensively in the graph theory literature. This parameter
has been studied on small grid graphs in \cite{r2018asymptotically}, and its density has been studied on infinite grid graphs in \cite{blessing2014tr,drews2017optimal}, on the King's lattice in \cite{Crepeau2019}, and on triangular finite and infinite graphs in \cite{harris2018broadcast}. Moreover, \cite{projectsInBroadcastDom} features many open problems for future study in this area.

In this paper, we expand on the study of domination by introducing a new domination variant which we henceforth refer to as \emph{directed $(t,r)$ broadcast domination}. Colloquially, the directed $(t,r)$ broadcast domination number is found by taking a directed graph and repeating the standard $(t,r)$ broadcast domination process, but we only allow signal to travel in the same direction as each of the graph's arcs. Note that this means that signal cannot travel ``both ways'' as it was previously able to over an edge in the undirected case. Hence, it is necessarily harder for signal to propagate from the towers toward other vertices within the graph. An application of directed $(t,r)$ broadcast domination arises by asking questions of standard $(t,r)$ broadcast domination in networks on which ``resources'' travel only in a single direction. Such settings include irrigation systems, electrical circuits, and regions of cities which are dense in one-way streets. We emphasize that the introduction of directed $(t,r)$ broadcast domination is a major contribution of this paper, which is the first to define the concept and study results arising from this new graph parameter. Moreover, we remark that all results in this paper provide a new direction for research, thereby opening many new avenues for future work in this area. 

\subsection{Overview} \label{subsec:overview}
The remainder of this paper is organized as follows. 
We begin with Section~\ref{subsec:tech_def}, which provides some important and technical definitions along with the necessary notation to make our understanding of domination theory and $(t,r)$ broadcast domination precise. In Section~\ref{chap:broadcast_domination_interval} we investigate the directed $(t,r)$ broadcast domination interval ($\interval{t,r}$), which is the smallest interval of integers containing the $(t,r)$ broadcast domination number of each orientation of a graph $G$. 
We focus first on the $\interval{t,r}$ on arbitrary graphs (Theorems \ref{thm:fullnessPart2} and \ref{thm:fullness}). 
Next, we find intervals contained in the $\interval{t,r}$ of some small grid graphs $G_{m,n}$ (Proposition \ref{prop:GridProps}). 
Additionally, we prove tight bounds on the $\interval{t,r}$ for the star graph $S_n$ (Propositions \ref{prop:firstStarProp}, \ref{prop:secondStarProp}, \ref{prop:thirdStarProp}) and the fullness of this interval (Theorem \ref{thm:starFullness}). Then, in Section~\ref{chap:infinite_grid} we introduce the density of directed $(t,r)$ broadcast domination on the infinite grid, focusing on the case $(t,r) = (2,2)$, where we give examples of orientations of the infinite grid which achieve efficient tower density $\frac{1}{3}$, $\frac{1}{2}$, and $\frac{2}{3}$ (Theorem \ref{thm:one_third_two_third}). 
The paper concludes with Section~\ref{chap:futureDirections}, which provides numerous directions for further research in directed $(t,r)$ broadcast domination.

\section{Technical Definitions and Background}\label{subsec:tech_def}

In this paper we follow the notation presented in Blessing~et.~al~\cite{blessing2014tr}. In this section we introduce all of the background needed to make our approach precise. For the interested reader we recommend Chartrand and Zhang's text~\cite{introToGraphTheory} as a good resource for further background in graph theory. We now begin with some basic graph theory concepts.

A \textbf{graph} $G$ is an ordered pair $(V(G), E(G))$ where $V(G)$ is a set of vertices and $E(G)$ is a multi-set of edges. An \textbf{edge} is a subset of $V(G)$ of size one or two.
A \textbf{directed graph (digraph)} is a graph whose edges are ordered pairs of vertices. To differentiate between directed and undirected edges, we henceforth refer to directed edges as \textbf{arcs}. 
\emph{Parallel arcs} refers to a set of multiple arcs starting and ending at the same pair of vertices, and a \emph{loop} is an arc that begins and ends at the same vertex.
A \textbf{simple digraph} is a digraph without parallel arcs or loops. 
A vertex $v$ is \textbf{incident} to an edge $e$ if $v \in e$. Two vertices $v_1, v_2$ are \textbf{adjacent} if there exists an edge $e$ such that $v_1, v_2 \in e$.
Lastly, an \textbf{oriented graph} is a simple digraph that does not contain doubly directed arcs, meaning that if $(u,v)$ is the arc from vertex $u$ to $v$, then there is no arc $(v,u)$ from $v$ to $u$. In other words, an oriented graph is a simple digraph where each edge is assigned a specified orientation.

We now introduce formally the concept of (undirected) $(t,r)$ broadcast domination. 
We follow the conventions and notation of Blessing~et.~al \cite{blessing2014tr}.

\begin{defn} Throughout we let $t \in \mathbb{N} := \{1, 2, 3, \ldots \}$. Given two vertices $u, v \in V(G)$, the \textbf{distance} between $u$ and $v$, denoted $d(u,v)$, is the minimum length of the paths connecting $u$ and $v$. We say that $v \in V(G)$ is a \textbf{broadcasting vertex of transmission strength $t$} if $u$ transmits a \emph{signal} of strength $t - d(u,v)$ to every vertex $u \in V(G)$ with $d(u,v) < t$. In the case where $d(u,v)\geq t$, then the broadcast vertex of transmission strength $t$ does not transmit any signal to the vertex $u$.
\end{defn}

\begin{defn}
Given a vertex $v$ and integer $t$, the \textbf{distance $t$ neighborhood of $v$} is defined as $N_t(v) = \{w \in V(G): d(w,v) < t\}$. If $v$ is selected to be a broadcasting vertex, then we call $N_t(v)$ the \textbf{broadcasting neighborhood of $v$}. Given a set of broadcast vertices $S \subseteq V(G)$, each with transmission strength $t$, we say that the \emph{reception} at vertex $w \in V(G)$ is 
$$r(w) = \underset{v \in S \cap N_t(w)}{\sum} (t - d(w,v)).$$
That is, the reception $r(w)$ is the sum of transmissions from neighboring broadcast vertices in $S$.
\end{defn}

\begin{defn}
A set $S \subseteq V(G)$ is called a \textbf{$(t,r)$ broadcast dominating set} if every vertex $w \in V(G)$ has a reception strength $r(w) \geq r$. For a finite graph $G$, the minimal cardinality among all broadcast dominating sets of $G$ is called the \textbf{$(t,r)$ broadcast domination number} of $G$ and is denoted $\gamma_{t,r}(G)$.
\end{defn}

\begin{ex}
The following is an example of a $(3,2)$ broadcast dominating set of an undirected 3-by-5 grid graph. To convince oneself that this set achieves the $(3,2)$ broadcast domination number requires proving that no sets of size 2 can sufficiently dominate the graph. One way of proving this claim could be to observe that all six vertices in the leftmost two columns cannot simultaneously receive sufficient reception unless at least two towers are placed somewhere within the leftmost three columns. Then, one can observe that such a placement would necessarily leave a vertex in the rightmost column with insufficient reception.
\end{ex}

\begin{figure}[h!]
\centering
\resizebox{2in}{!}{
\begin{tikzpicture}[vertex_style/.style={circle,ball color=white,draw=black!10!white},
edge_style/.style={ thick, black, drop shadow={opacity=0.4}},vertex_style_red/.style={circle,ball color=red,draw=red!40!white},
edge_style/.style={thick, black}]
\tikzstyle{ghost node}=[draw=none]
\tikzset{black node/.style={circle,draw=black, fill=white, inner sep=10.5}}
\tikzset{red node/.style={circle,draw=black, fill=red, inner sep=10.5}}
\foreach \x in {1, 2, 3, 4, 5}
    \foreach \y in {1, 2, 3} 
        \node [vertex_style,minimum size=.5in] (\x\y) at (2*\x,2*\y){0};
        \node [vertex_style,minimum size=.5in] (11) at (2*1,2*1){1 + 1};
        \node [vertex_style,minimum size=.5in] (12) at (2*1,2*2){2};
        \node [vertex_style_red,minimum size=.5in] (13) at (2*1,2*3){3};
        \node [vertex_style,minimum size=.5in] (21) at (2*2,2*1){2};
        \node [vertex_style,minimum size=.5in] (22) at (2*2,2*2){1 + 1};
        \node [vertex_style,minimum size=.5in] (23) at (2*2,2*3){2};
        \node [vertex_style_red,minimum size=.5in] (31) at (2*3,2*1){3};
        \node [vertex_style,minimum size=.5in] (32) at (2*3,2*2){2};
        \node [vertex_style,minimum size=.5in] (33) at (2*3,2*3){1+1+1};
        \node [vertex_style,minimum size=.5in] (41) at (2*4,2*1){2};
        \node [vertex_style,minimum size=.5in] (42) at (2*4,2*2){1+1};
        \node [vertex_style,minimum size=.5in] (43) at (2*4,2*3){2};
        \node [vertex_style,minimum size=.5in] (51) at (2*5,2*1){1+1};
        \node [vertex_style,minimum size=.5in] (52) at (2*5,2*2){2};
        \node [vertex_style_red,minimum size=.5in] (53) at (2*5,2*3){3};
\foreach \x in {1,2,3,4,5}
    \foreach \y  [count=\yi from 2] in {1,2} 
        \path[] (\x\y)edge(\x\yi);
\foreach \y in {1,2,3}
    \foreach \x  [count=\xi from 2] in {1,2,3,4}
        \path[] (\x\y)edge(\xi\y);
\end{tikzpicture}}
\caption{An example of a minimal $(3,2)$ dominating set of the 3-by-5 grid graph is shown. Towers are highlighted in red, and values/sums within vertices denote the reception at those vertices.}
    \label{fig:3x5solutionExample}
\end{figure}
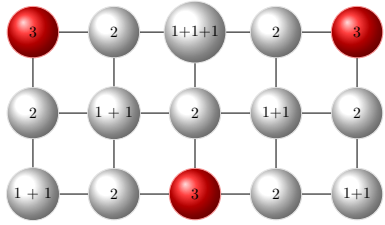

\begin{ex}
For a given graph, there may be many possible broadcast dominating sets of minimal cardinality. Figure~\ref{fig:3x5solutionExample(3,1)a} illustrates two examples of $(3,1)$ broadcast dominating sets with cardinality equal to $2$, which is the $(3,1)$ broadcast domination number of the 3-by-5 grid graph.
\end{ex}

\begin{figure}[h!]
\centering
\resizebox{2in}{!}{
\begin{tikzpicture}[vertex_style/.style={circle,ball color=white,draw=black!10!white},
edge_style/.style={ thick, black, drop shadow={opacity=0.4}},vertex_style_red/.style={circle,ball color=red,draw=red!40!white},
edge_style/.style={thick, black,drop shadow={opacity=0.4}}]
\tikzstyle{ghost node}=[draw=none]
\tikzset{black node/.style={circle,draw=black, fill=white, inner sep=10.5}}
\tikzset{red node/.style={circle,draw=black, fill=red, inner sep=10.5}}

\foreach \x in {1, 2, 3, 4, 5}
    \foreach \y in {1, 2, 3} 
        \node [vertex_style,minimum size=.5in] (\x\y) at (2*\x,2*\y){0};
        \node [vertex_style,minimum size=.5in] (11) at (2*1,2*1){1};
        \node [vertex_style,minimum size=.5in] (12) at (2*1,2*2){2};
        \node [vertex_style,minimum size=.5in] (13) at (2*1,2*3){1};
        \node [vertex_style,minimum size=.5in] (21) at (2*2,2*1){2};
        \node [vertex_style_red,minimum size=.5in] (22) at (2*2,2*2){3};
        \node [vertex_style,minimum size=.5in] (23) at (2*2,2*3){2};
        \node [vertex_style,minimum size=.5in] (31) at (2*3,2*1){1 + 1};
        \node [vertex_style,minimum size=.5in] (32) at (2*3,2*2){2 + 2};
        \node [vertex_style,minimum size=.5in] (33) at (2*3,2*3){1 + 1};
        \node [vertex_style,minimum size=.5in] (41) at (2*4,2*1){2};
        \node [vertex_style_red,minimum size=.5in] (42) at (2*4,2*2){3};
        \node [vertex_style,minimum size=.5in] (43) at (2*4,2*3){2};
        \node [vertex_style,minimum size=.5in] (51) at (2*5,2*1){1};
        \node [vertex_style,minimum size=.5in] (52) at (2*5,2*2){2};
        \node [vertex_style,minimum size=.5in] (53) at (2*5,2*3){1};
\foreach \x in {1,2,3,4,5}
    \foreach \y  [count=\yi from 2] in {1,2} 
        \path[] (\x\y)edge(\x\yi);
\foreach \y in {1,2,3}
    \foreach \x  [count=\xi from 2] in {1,2,3,4}
        \path[] (\x\y)edge(\xi\y);
\end{tikzpicture}}
\hspace{1cm}
\resizebox{2in}{!}{
\begin{tikzpicture}[vertex_style/.style={circle,ball color=white,draw=black!10!white},
edge_style/.style={ thick, black, drop shadow={opacity=0.4}},vertex_style_red/.style={circle,ball color=red,draw=red!40!white},
edge_style/.style={thick, black,drop shadow={opacity=0.4}}]
\tikzstyle{ghost node}=[draw=none]
\tikzset{black node/.style={circle,draw=black, fill=white, inner sep=10.5}}
\tikzset{red node/.style={circle,draw=black, fill=red, inner sep=10.5}}

\foreach \x in {1, 2, 3, 4, 5}
    \foreach \y in {1, 2, 3} 
        \node [vertex_style,minimum size=.5in] (\x\y) at (2*\x,2*\y){0};
        \node [vertex_style,minimum size=.5in] (11) at (2*1,2*1){1};
        \node [vertex_style,minimum size=.5in] (12) at (2*1,2*2){2};
        \node [vertex_style,minimum size=.5in] (13) at (2*1,2*3){1};
        \node [vertex_style,minimum size=.5in] (21) at (2*2,2*1){2};
        \node [vertex_style_red,minimum size=.5in] (22) at (2*2,2*2){3};
        \node [vertex_style,minimum size=.5in] (23) at (2*2,2*3){2};
        \node [vertex_style,minimum size=.5in] (31) at (2*3,2*1){1};
        \node [vertex_style,minimum size=.5in] (32) at (2*3,2*2){2 + 1};
        \node [vertex_style,minimum size=.5in] (33) at (2*3,2*3){1};
        \node [vertex_style,minimum size=.5in] (41) at (2*4,2*1){1};
        \node [vertex_style,minimum size=.5in] (42) at (2*4,2*2){2};
        \node [vertex_style,minimum size=.5in] (43) at (2*4,2*3){1};
        \node [vertex_style,minimum size=.5in] (51) at (2*5,2*1){2};
        \node [vertex_style_red,minimum size=.5in] (52) at (2*5,2*2){3};
        \node [vertex_style,minimum size=.5in] (53) at (2*5,2*3){2};
\foreach \x in {1,2,3,4,5}
    \foreach \y  [count=\yi from 2] in {1,2} 
        \path[] (\x\y)edge(\x\yi);
\foreach \y in {1,2,3}
    \foreach \x  [count=\xi from 2] in {1,2,3,4}
        \path[] (\x\y)edge(\xi\y);
\end{tikzpicture}}
\caption{Two examples of minimal $(3,1)$ dominating sets of the 3-by-5 grid graph.} 
\label{fig:3x5solutionExample(3,1)a}
\end{figure}
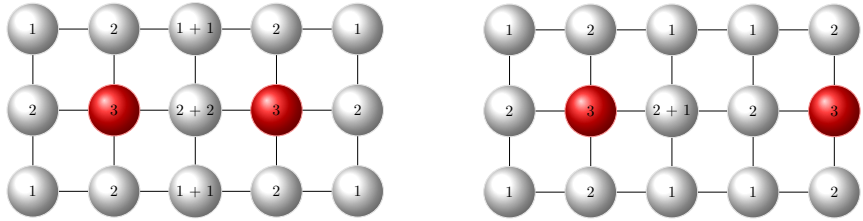

It is important to note that the dominating sets presented in Figure~\ref{fig:3x5solutionExample} form only a subset of all of the possible $(3,1)$ dominating sets of the 3-by-5 grid graph. 

\subsection{Directed \texorpdfstring{$(t,r)$}{(t,r)} Broadcast Domination}\label{subsec:directed_t_r}

Now, we introduce \emph{directed $(t,r)$ broadcast domination}, the main object of study in this paper. The following are previous definitions which we extend to an orientation $\vec{G}$ of $G$. That is, we orient all of the edges of $G$, letting $\vec{G}$ be the resulting digraph, and define $(t,r)$ broadcast domination on~$\vec{G}$. 

\begin{defn}
Let $t \in \mathbb{N} := \{1, 2, 3, \ldots \}$. Given two vertices $u, v \in V(\vec{G})$ the distance from $u$ to $v$, denoted $d(u,v)$, is the minimum length of the directed paths from $u$ to $v$. We say that $u \in V(\vec{G})$ is a \textbf{broadcasting vertex of directed transmission strength $t$} if it transmits a \textbf{signal} of strength $t - d(u,v)$ to every vertex $v \in V(G)$ with $d(u,v) < t$. Once again, vertices which are distance at least $t$ from $u$ receive signal of strength 0 from $u$.
\end{defn}

Note that in directed graphs it is not necessarily the case that $d(u,v) = d(v,u)$. Also note that if there is no directed path from $u$ to $v$, then $d(u,v) = \infty$.
Furthermore, we adopt the following two conventions. Firstly, the transmission strength is omitted when it is clear from context. Secondly, this paper uses the terms ``broadcasting vertex" and ``tower" interchangeably. Now we continue with other important definitions.

\begin{defn}
Given a vertex $v \in V(\vec{G})$ and positive integer $t$, the \textbf{distance $t$ out-neighborhood of $v$} is defined as $N_t^+(v) = \{w \in V(G): d(v,w) < t\}$. 
Likewise, the \textbf{distance $t$ in-neighborhood of $v$} is defined as $N_t^-(v) = \{w \in V(G): d(w,v) < t\}$. 
If $v$ is selected to be a broadcasting vertex, then we call $N_t^+(v)$ the \textbf{broadcasting out-neighborhood of $v$}. 
Given a set of broadcast vertices $S \subseteq V(\vec{G})$, each with transmission strength $t$, we say that the \textbf{directed reception} at vertex $w \in V(\vec{G})$ is 
$$\vec{r}(w) = \underset{v \in S \cap N_t^-(w)}{\sum} (t - d(v, w)).$$
That is, the directed reception $\vec{r}(w)$ is the sum of the out-transmissions from its in-neighboring broadcast vertices in $S$.
\end{defn}

\begin{defn}
A set $S \subseteq V(\vec{G})$ is called a \textbf{directed $(t,r)$ broadcast dominating set} if every vertex $w \in V(\vec{G})$ has a directed reception strength $\vec{r}(w) \geq r$. For a finite digraph $\vec{G}$, the minimal cardinality among all broadcast dominating sets of $\vec{G}$ is called the \textbf{directed $(t,r)$ broadcast domination number of} $\vec{G}$ and is denoted $\gamma_{t,r}(\vec{G})$.
\end{defn}

For any simple graph $G$ there are $2^{|E(G)|}$ orientations of the edges of $G$. Let $\orien{G}$ be the set of all orientations of the edges of the graph $G$. Then we are interested in determining possible values of $\gamma_{t,r}(\vec{G})$ whenever $\vec{G}\in\orien{G}$. This motivates the following.

\begin{defn}
Let $A_{t,r}(G) := \{\gamma_{t,r}(\vec{G}) \mid \vec{G}\in\orien{G} \}$. The \textbf{directed $(t,r)$ broadcast domination interval} of a graph $G$, denoted $\dbdi{t,r}{G}$, is the interval $[d, D]$, where $d = \min(A_{t,r}(G))$ and $D = \max(A_{t,r}(G))$. 
\end{defn}
Colloquially, the directed $(t,r)$ broadcast domination interval of $G$ is smallest contiguous interval of integers which contains all possible values of $\gamma_{t,r}(\vec{G})$ for $\vec{G}\in\orien{G}$. Such a concept is motivated by the case in which signal (or some other resource) may only travel in a single direction, but the user maintains control over that direction. For example, a factory manager may wish to establish a traffic pattern for a factory line in her facility, and she may want to know the discrepancies between different directions of travel throughout the plant. She may also wish to know how drastic such trade-offs could be, for which knowing the width of the $\interval{t,r}$ would be useful. A main result of this paper establishes that for any graph $G$ and
$(t,r)\in\{(2,2)\}\cup\{(t,1):t\geq 1\}$, 
given any any integer in $k\in\dbdi{t,r}{G}$, 
there exists $\vec{G}\in\orien{G}$, 
such that $\gamma_{t,r}(\vec{G})=k$. This is the content of Theorem \ref{thm:fullness}, which we prove in Section~\ref{chap:broadcast_domination_interval}.
We remark that the introduction of these novel definitions serves as the first contribution of this paper to the study of $(t,r)$ broadcast domination.

\section{The Directed \texorpdfstring{$(t,r)$}{(t,r)} Broadcast Domination Interval}
\label{chap:broadcast_domination_interval}

This section concerns the directed $(t,r)$ broadcast domination interval, first discussing arbitrary graphs before focusing on stars and small grid graphs.

\subsection{Results on Arbitrary Graphs}\label{subsec:arbitrary_graphs}

One interesting question surrounding the directed $(t,r)$ broadcast domination interval is whether the name \emph{interval} is actually appropriate. Namely, given a graph $G$ and fixed values $t$ and $r$, does $\dbdi{t,r}{G} = [d,D]$ contain all integer values within this interval. In other words, for all $b \in [d,D] \cap \Z$, does there exist an orientation $\vec{G}$ such that $\gamma_{t,r}(\vec{G}) = b$? In this section, we show that in the case $r = 1$ the answer is yes, and we conjecture this to be true for all $r$-values.
Our first result shows that for a fixed $t$ and for $r = 1$, flipping an arc in any orientation of a graph cannot change the directed $(t,r)$ broadcast domination number of the graph by more than 1.

\begin{lem} \label{lem: arcFlips}
Let $t$ be any positive integer. Given a graph $G$, let $\vec{G}_0$ and $\vec{G}_1$ be two orientations of $G$ where $\vec{G}_1$ can be obtained from $\vec{G}_0$ by flipping (reversing the direction of) a single arc. Then $|\gamma_{t,1}(\vec{G}_0) - \gamma_{t,1}(\vec{G}_1)| \leq 1$. \end{lem}
\begin{proof}
Without loss of generality, suppose $\vec{G}_0$ is the starting orientation, and let $a = (u,v)$ be the arc which is flipped to obtain $\vec{G}_1$ so that $a$ becomes $a' = (v,u)$. 
If $S \subseteq V(G)$ is a $(t,1)$ directed broadcast dominating set of $\vec{G}_0$ with $|S| = \gamma_{t,1}(\vec{G}_0)$, then $S \cup \{v\}$ is a $(t,1)$ directed broadcast dominating set of $\vec{G}_1$ as we have now ensured that $v$ has reception $t\geq 1$ since it is now a tower. Moreover, the size of a directed $(t,1)$ broadcast dominating set has increased by at most 1. 
Again, note that because $v$ is now a broadcasting vertex, the transmission strength at $v$ has strictly increased if $v$ was not previously a broadcasting vertex, or it has stayed the same if $v$ was already a broadcasting vertex. Thus, the directed reception at all other vertices is at least what it was using the dominating set $S$. This implies that  $\gamma_{t,1}(\vec{G}_1) \leq \gamma_{t,1}(\vec{G}_0) + 1$. A visual representation of this argument is shown in Figure~\ref{fig:arcFlipsProof}. To obtain the reverse inequality, suppose the directed $(t,r)$ broadcast domination number could decrease by more than 1 after flipping a single arc. 
This would imply that flipping that same arc once again would cause the directed $(t,r)$ broadcast domination number to increase by more than 1, a contradiction, since we assumed $|S|$ was the directed $(t,1)$ broadcast domination number of the initial orientation $\vec{G}_0$.
\end{proof}

\begin{figure}[h!]
\centering
\resizebox{1.25in}{!}{
\begin{tikzpicture}
[vertex_style/.style={circle,ball color=white,draw=black!10!white},edge_style/.style={ thick, black,drop shadow={opacity=0.4}},vertex_style_red/.style={circle,ball color=red,draw=red!40!white},vertex_style_blue/.style={circle,ball color=blue,draw=blue!40!white,drop shadow={opacity=0.2}},edge_style/.style={thick, black,drop shadow={opacity=0.4}},rotate=18]
\node[vertex_style,minimum size=.5in](a0) at (0,0) {$t-1$};
\node[vertex_style_red,minimum size=.5in](a1) at (3,0) {$t$};
\node[vertex_style,minimum size=.5in](a2) at (2.1,2.1) {$t-2$};
\node[vertex_style,minimum size=.5in](a3) at (0,3) {$t-2$};
\node[vertex_style,minimum size=.5in](a4) at (-2.1,2.1) {$t-2$};
\node[vertex_style,minimum size=.5in](a5) at (-3,0) {$t-2$};
\node[vertex_style,minimum size=.5in](a6) at (-2.1,-2.1) {$t-2$};
\node[vertex_style,minimum size=.5in](a7) at (0,-3) {$t-2$};
\node[vertex_style,minimum size=.5in](a8) at (2.1,-2.1) {$t-2$};
\draw[<-, ultra thick] (a0)--(a1);
\draw[->, ultra thick] (a0)--(a2);
\draw[->, ultra thick] (a0)--(a3);
\draw[->, ultra thick] (a0)--(a4);
\draw[->, ultra thick] (a0)--(a5);
\draw[->, ultra thick] (a0)--(a6);
\draw[->, ultra thick] (a0)--(a7);
\draw[->, ultra thick] (a0)--(a8);
\end{tikzpicture}}
\hspace{1cm}
\resizebox{1.25in}{!}{
\begin{tikzpicture}
[vertex_style/.style={circle,ball color=white,draw=black!10!white},edge_style/.style={ thick, black,drop shadow={opacity=0.4}},vertex_style_red/.style={circle,ball color=red,draw=red!40!white},vertex_style_blue/.style={circle,ball color=blue,draw=blue!40!white,drop shadow={opacity=0.2}},edge_style/.style={thick, black,drop shadow={opacity=0.4}},rotate=18]
\node[vertex_style_red,minimum size=.5in](a0) at (0,0) {$t$};
\node[vertex_style_red,minimum size=.5in](a1) at (3,0) {$2t-1$};
\node[vertex_style,minimum size=.5in](a2) at (2.1,2.1) {$t-1$};
\node[vertex_style,minimum size=.5in](a3) at (0,3) {$t-1$};
\node[vertex_style,minimum size=.5in](a4) at (-2.1,2.1) {$t-1$};
\node[vertex_style,minimum size=.5in](a5) at (-3,0) {$t-1$};
\node[vertex_style,minimum size=.5in](a6) at (-2.1,-2.1) {$t-1$};
\node[vertex_style,minimum size=.5in](a7) at (0,-3) {$t-1$};
\node[vertex_style,minimum size=.5in](a8) at (2.1,-2.1) {$t-1$};
\draw[->, ultra thick, color = green] (a0)--(a1);
\draw[->, ultra thick] (a0)--(a2);
\draw[->, ultra thick] (a0)--(a3);
\draw[->, ultra thick] (a0)--(a4);
\draw[->, ultra thick] (a0)--(a5);
\draw[->, ultra thick] (a0)--(a6);
\draw[->, ultra thick] (a0)--(a7);
\draw[->, ultra thick] (a0)--(a8);
\end{tikzpicture}}
\caption{The graphs $\vec{G_0}$ and $\vec{G_1}$ are shown on the left and right, respectively, with the flipped arc highlighted in green. Because the tail of the flipped arc becomes a broadcasting vertex, its transmission strength strictly increases, even if it receives signal from other towers.}
\label{fig:arcFlipsProof}
\end{figure}
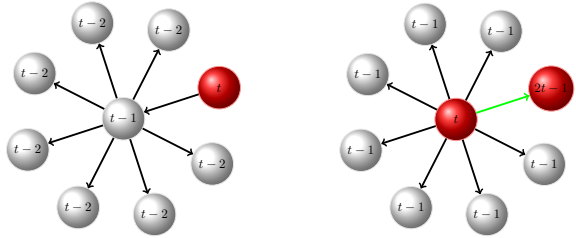

\begin{thm}
\label{thm:fullness}
Given a graph $G$, let $\dbdi{t,1}{G} = [d, D]$. Then, for every $b \in \Z \cap [d, D]$, there exists an orientation $\vec{G}_b$ with $\gamma_{t,1}(\vec{G}_b) = b$.
\end{thm}
\begin{proof}
Let $\dbdi{t,1}{G} = [d, D]$. Let $\vec{G}_d, \vec{G}_D$ be the orientations of $G$ which achieve directed $(t,1)$ broadcast domination numbers $d$ and $D$, respectively. Notice that $\vec{G}_D$ can be obtained from $\vec{G}_d$ by a series of arc flips. 
As we perform these arc flips, the directed $(t,1)$ broadcast domination number of the resulting graph can never change by more than one, as we established in Lemma \ref{lem: arcFlips}. This implies that at some point each integer value in $\Z \cap [d, D]$ was attained.
\end{proof}

Interestingly, an identical argument, as that presented in Theorem \ref{thm:fullness}, also holds in the case $(t,r) = (2,2)$. This is because signal from any broadcasting vertex reaches only that broadcasting vertex's immediate out-neighborhood. As a result, we have the following.

\begin{thm}
Given a graph $G$, let $\dbdi{2,2}{G} = [d, D]$. Then, for every $b \in \Z \cap [d, D]$, there exists an orientation $\vec{G}_b$ with $\gamma_{2,2}(\vec{G}_b) = b$.
\end{thm}
\label{thm:fullnessPart2}
\begin{proof}
The argument is identical to that of Lemma~\ref{lem: arcFlips} and Theorem~\ref{thm:fullness}. 
Consider graphs $\vec{G_0}$ and $\vec{G_1}$ as before.
Because signal form any broadcasting vertex reaches only the out-neighborhood of that vertex, flipping such an arc causes at most one vertex to lose sufficient reception. 
To mitigate this, include that vertex in the set of broadcasting vertices for $G_1$, and the resulting set forms a directed $(2,2)$ broadcast dominating set. 
As a result, $\gamma_{2,2}(\vec{G}_1) - \gamma_{2,2}(\vec{G}_0) \leq 1$, and by symmetry we have that  $|\gamma_{2,2}(\vec{G}_0) - \gamma_{2,2}(\vec{G}_1) | \leq 1$. This immediately proves the desired result.
\end{proof}

It is important to illustrate, however, that the same argument does not hold in general for arbitrary $t$ and $r$. Figure~\ref{fig:flippingArgDoesn'tScale} shows that on a path on seven vertices with $(t,r) = (10,8)$, attempting to make the vertex which becomes the new tail of the flipped arc a broadcasting vertex does not create a directed $(t,r)$ broadcast dominating set.

\begin{figure}[h!]
\centering
\resizebox{2.5in}{!}{
\begin{tikzpicture}[vertex_style/.style={circle,ball color=white,draw=black!10!white},
edge_style/.style={ thick, black, drop shadow={opacity=0.4}},vertex_style_red/.style={circle,ball color=red,draw=red!40!white},
edge_style/.style={thick, black,drop shadow={opacity=0.4}}]
\tikzstyle{ghost node}=[draw=none]
\tikzset{black node/.style={circle,draw=black, fill=white, inner sep=10.5}}
\tikzset{red node/.style={circle,draw=black, fill=red, inner sep=10.5}}

\foreach \x in {1, 2, 3, 4, 5, 6, 7}
    \foreach \y in {1, 2} 
        \node [vertex_style,minimum size=.5in] (\x\y) at (2*\x,2*\y){0};
        \node [vertex_style,minimum size=.5in] (11) at (2*1,2*1){6};
        \node [vertex_style,minimum size=.5in] (12) at (2*1,2*2){5+4};
        \node [vertex_style,minimum size=.5in] (21) at (2*2,2*1){7};
        \node [vertex_style,minimum size=.5in] (22) at (2*2,2*2){6+5};
        \node [vertex_style,minimum size=.5in] (31) at (2*3,2*1){8};
        \node [vertex_style,minimum size=.5in] (32) at (2*3,2*2){7+6};
        \node [vertex_style,minimum size=.5in] (41) at (2*4,2*1){9};
        \node [vertex_style,minimum size=.5in] (42) at (2*4,2*2){8+7};
        \node [vertex_style_red,minimum size=.5in] (51) at (2*5,2*1){10};
        \node [vertex_style,minimum size=.5in] (52) at (2*5,2*2){9+8};
        \node [vertex_style_red,minimum size=.5in] (61) at (2*6,2*1){9+9};
        \node [vertex_style_red,minimum size=.5in] (62) at (2*6,2*2){10+9};
        \node [vertex_style_red,minimum size=.5in] (71) at (2*7,2*1){10};
        \node [vertex_style_red,minimum size=.5in] (72) at (2*7,2*2){10};
\foreach \y in {1,2}
    \foreach \x  [count=\xi from 2] in {1,2,3,4,5,6}
        \path[] (\x\y)edge(\xi\y);
\draw[<-, ultra thick] (11)--(21);
\draw[<-, ultra thick] (21)--(31);
\draw[<-, ultra thick] (31)--(41);
\draw[<-, ultra thick] (41)--(51);
\draw[->, ultra thick, color = green] (51)--(61);
\draw[<-, ultra thick] (61)--(71);
\draw[<-, ultra thick] (12)--(22);
\draw[<-, ultra thick] (22)--(32);
\draw[<-, ultra thick] (32)--(42);
\draw[<-, ultra thick] (42)--(52);
\draw[<-, ultra thick] (52)--(62);
\draw[<-, ultra thick] (62)--(72);
\end{tikzpicture}}
\caption{For $(t,r) = (10,8)$, the strategy used in Lemma~\ref{lem: arcFlips} does not hold. In the following example, the resulting set of vertices of the bottom graph does not form a directed $(10,8)$ dominating set. Broadcasting vertices are highlighted in red, and the arc which was flipped is highlighted in green.}
\label{fig:flippingArgDoesn'tScale}
\end{figure}
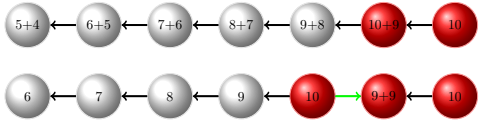

Let $\vec{G}^T$ denote the \textbf{transpose} orientation of $\vec{G}$, which orients every arc in $\vec{G}$ in the opposite direction in $\vec{G}^T$. The following is another known result from 2013 about digraph parameters. While it appears to be useful in proving the fullness of the directed $(t,r)$ broadcast domination interval for arbitrary $t$ and $r$, the first condition is often not true for the directed $(t,r)$ broadcast domination number of a graph's orientation.

\begin{thm}[Theorem~2.1, \cite{ISUREU13}] \label{thm:fullnessISUREU}
Suppose $\beta$ is an integer-valued digraph parameter with the following properties for every oriented graph $\vec{G}$:
\begin{enumerate}
    \item $\beta(\vec{G}) = \beta(\vec{G}^T)$.
    \item If $(u,v) \in E(\vec{G})$ and $\vec{G_0}$ is obtained from $\vec{G}$ by replacing $(u,v)$ by $(v,u)$ (i.e., reversing the orientation of one arc), then $|\beta(\vec{G}) - \beta(\vec{G_0})| \leq 1$.
\end{enumerate}
Then for any two orientations $\vec{G_1}$ and $\vec{G_2}$ of the same graph $G$, $|\beta(\vec{G_2}) - \beta(\vec{G_1})| \leq \floor{\frac{|E(G)|}{2}}$.
\end{thm}

For the directed $(t,r)$ broadcast domination parameter, failure of the first condition in Theorem \ref{thm:fullnessISUREU} is perhaps most easily seen on the star on $n$ vertices, shown in Figure~\ref{fig:starFailingCondition}.

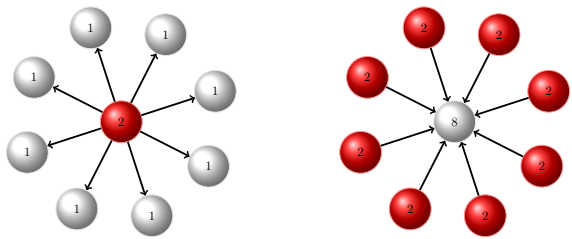
\begin{figure}[h!]
\centering
\resizebox{1.25in}{!}{
\begin{tikzpicture}
[vertex_style/.style={circle,ball color=white,draw=black!10!white},edge_style/.style={ thick, black,drop shadow={opacity=0.4}},vertex_style_red/.style={circle,ball color=red,draw=red!40!white},vertex_style_blue/.style={circle,ball color=blue,draw=blue!40!white,drop shadow={opacity=0.2}},edge_style/.style={thick, black,drop shadow={opacity=0.4}},rotate=18]
\node[vertex_style_red,minimum size=.5in](a0) at (0,0) {2};
\node[vertex_style,minimum size=.5in](a1) at (3,0) {1};
\node[vertex_style,minimum size=.5in](a2) at (2.1,2.1) {1};
\node[vertex_style,minimum size=.5in](a3) at (0,3) {1};
\node[vertex_style,minimum size=.5in](a4) at (-2.1,2.1) {1};
\node[vertex_style,minimum size=.5in](a5) at (-3,0) {1};
\node[vertex_style,minimum size=.5in](a6) at (-2.1,-2.1) {1};
\node[vertex_style,minimum size=.5in](a7) at (0,-3) {1};
\node[vertex_style,minimum size=.5in](a8) at (2.1,-2.1) {1};
\draw[->, ultra thick] (a0)--(a1);
\draw[->, ultra thick] (a0)--(a2);
\draw[->, ultra thick] (a0)--(a3);
\draw[->, ultra thick] (a0)--(a4);
\draw[->, ultra thick] (a0)--(a5);
\draw[->, ultra thick] (a0)--(a6);
\draw[->, ultra thick] (a0)--(a7);
\draw[->, ultra thick] (a0)--(a8);
\end{tikzpicture}
}
\hspace{1cm}
\resizebox{1.25in}{!}{
\begin{tikzpicture}
[vertex_style/.style={circle,ball color=white,draw=black!10!white},edge_style/.style={ thick, black,drop shadow={opacity=0.4}},vertex_style_red/.style={circle,ball color=red,draw=red!40!white},vertex_style_blue/.style={circle,ball color=blue,draw=blue!40!white,drop shadow={opacity=0.2}},edge_style/.style={thick, black,drop shadow={opacity=0.4}},rotate=18]
\node[vertex_style,minimum size=.5in](a0) at (0,0) {8};
\node[vertex_style_red,minimum size=.5in](a1) at (3,0) {2};
\node[vertex_style_red,minimum size=.5in](a2) at (2.1,2.1) {2};
\node[vertex_style_red,minimum size=.5in](a3) at (0,3) {2};
\node[vertex_style_red,minimum size=.5in](a4) at (-2.1,2.1) {2};
\node[vertex_style_red,minimum size=.5in](a5) at (-3,0) {2};
\node[vertex_style_red,minimum size=.5in](a6) at (-2.1,-2.1) {2};
\node[vertex_style_red,minimum size=.5in](a7) at (0,-3) {2};
\node[vertex_style_red,minimum size=.5in](a8) at (2.1,-2.1) {2};
\draw[<-, ultra thick] (a0)--(a1);
\draw[<-, ultra thick] (a0)--(a2);
\draw[<-, ultra thick] (a0)--(a3);
\draw[<-, ultra thick] (a0)--(a4);
\draw[<-, ultra thick] (a0)--(a5);
\draw[<-, ultra thick] (a0)--(a6);
\draw[<-, ultra thick] (a0)--(a7);
\draw[<-, ultra thick] (a0)--(a8);
\end{tikzpicture}
 }
\caption{For $(t,r) = (2,1)$, the digraph $\vec{G}$ on the left has domination number $1$, while the digraph $\vec{G}^T$ on the right has domination number $8$. Broadcasting vertices are highlighted in red.}
\label{fig:starFailingCondition}
\end{figure}

Moreover, for some graphs and given some positive integers $t$ and $r$, flipping an arc may change the $(t,r)$ broadcast domination number by more than 1, making Lemma~\ref{lem: arcFlips} not possible to generalize for arbitrary $t$ and $r$. An example of an instance where an arc flip changes the directed $(5,3)$ broadcast domination number of a graph from two to four can be seen in Figure~\ref{fig:flipIncreaseBy2}.

\begin{figure}[h!]
\centering
\resizebox{2in}{!}{
\begin{tikzpicture}[vertex_style/.style={circle,ball color=white,draw=black!10!white},
edge_style/.style={ thick, black, drop shadow={opacity=0.4}},vertex_style_red/.style={circle,ball color=red,draw=red!40!white},
edge_style/.style={thick, black,drop shadow={opacity=0.4}}]
\tikzstyle{ghost node}=[draw=none]
\tikzset{black node/.style={circle,draw=black, fill=white, inner sep=10.5}}
\tikzset{red node/.style={circle,draw=black, fill=red, inner sep=10.5}}

\node [vertex_style,minimum size=.5in] (12) at (2*1,2*2){2+1};
\node [vertex_style,minimum size=.5in] (21) at (2*2,2*1){2+1};
\node [vertex_style,minimum size=.5in] (22) at (2*2,2*2){3+2};
\node [vertex_style,minimum size=.5in] (23) at (2*2,2*3){2+1};
\node [vertex_style,minimum size=.5in] (31) at (2*3,2*1){3+2};
\node [vertex_style,minimum size=.5in] (32) at (2*3,2*2){4+3};
\node [vertex_style,minimum size=.5in] (33) at (2*3,2*3){3+2};
\node [vertex_style,minimum size=.5in] (41) at (2*4,2*1){4+3};
\node [vertex_style_red,minimum size=.5in] (42) at (2*4,2*2){5+4};
\node [vertex_style,minimum size=.5in] (43) at (2*4,2*3){4+3};
\node [vertex_style_red,minimum size=.5in] (52) at (2*5,2*2){5};
\draw [<-, ultra thick] (12)--(22);
\draw [<-, ultra thick] (22)--(32);
\draw [<-, ultra thick] (32)--(42);
\draw [<-, ultra thick] (42)--(52);
\draw [<-, ultra thick] (23)--(33);
\draw [<-, ultra thick] (33)--(43);
\draw [<-, ultra thick] (21)--(31);
\draw [<-, ultra thick] (31)--(41);
\draw [<-, ultra thick] (41)--(42);
\draw [->, ultra thick] (42)--(43);
\end{tikzpicture}}
\hspace{1cm}
\resizebox{2in}{!}{
\begin{tikzpicture}[vertex_style/.style={circle,ball color=white,draw=black!10!white},
edge_style/.style={ thick, black, drop shadow={opacity=0.4}},vertex_style_red/.style={circle,ball color=red,draw=red!40!white},
edge_style/.style={thick, black,drop shadow={opacity=0.4}}]
\tikzstyle{ghost node}=[draw=none]
\tikzset{black node/.style={circle,draw=black, fill=white, inner sep=10.5}}
\tikzset{red node/.style={circle,draw=black, fill=red, inner sep=10.5}}

\node [vertex_style,minimum size=.5in] (12) at (2*1,2*2){3+2};
\node [vertex_style,minimum size=.5in] (21) at (2*2,2*1){3+2};
\node [vertex_style,minimum size=.5in] (22) at (2*2,2*2){4+3};
\node [vertex_style,minimum size=.5in] (23) at (2*2,2*3){3+2};
\node [vertex_style,minimum size=.5in] (31) at (2*3,2*1){4+3};
\node [vertex_style_red,minimum size=.5in] (32) at (2*3,2*2){5+4};
\node [vertex_style,minimum size=.5in] (33) at (2*3,2*3){4+3};
\node [vertex_style_red,minimum size=.5in] (41) at (2*4,2*1){5+4};
\node [vertex_style_red,minimum size=.5in] (42) at (2*4,2*2){5};
\node [vertex_style_red,minimum size=.5in] (43) at (2*4,2*3){5+4};
\node [vertex_style,minimum size=.5in] (52) at (2*5,2*2){4};
\draw [<-, ultra thick] (12)--(22);
\draw [<-, ultra thick] (22)--(32);
\draw [<-, ultra thick] (32)--(42);
\draw [->, ultra thick, color = green] (42)--(52);
\draw [<-, ultra thick] (23)--(33);
\draw [<-, ultra thick] (33)--(43);
\draw [<-, ultra thick] (21)--(31);
\draw [<-, ultra thick] (31)--(41);
\draw [<-, ultra thick] (41)--(42);
\draw [->, ultra thick] (42)--(43);
\end{tikzpicture}}
\caption{Flipping the rightmost arc of the graph on the left increases the directed $(5,3)$ broadcast domination number from 2 to 4 (an increase of more than 1). Showing Lemma~\ref{lem: arcFlips} does not hold for $r>1$ in general. The graph on the right depicts the resultant graph and dominating set, with the flipped arc highlighted in green.}
\label{fig:flipIncreaseBy2}
\end{figure}
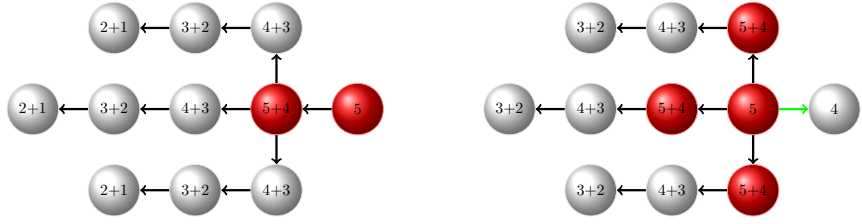

\subsection{Results on Small Grid Graphs} \label{subsec:small_grid}

In this section, we turn our attention to the directed $(t,r)$ broadcast domination interval on small grid graphs, a family of graphs discussed in \cite{blessing2014tr}. Before doing so, we begin with the definition of a grid graph and an observation about the relationship between the standard $(t,r)$ broadcast domination number of a graph $G$ and the directed $(t,r)$ broadcast domination number of $\vec{G}$, an orientation of that graph. 

Before defining the grid graph, we recall that the Cartesian product of graphs $G$ and $H$, denoted $G \Box H$, has vertex set $V(G) \times V(H)$, and it has edge set 
\[
E(G \Box H) = \left\{ \{(u_1, v_1), (u_2, v_2) \} \text{ such that } 
\begin{Bmatrix}
u_1 = u_2 \text{ and } \{v_1, v_2\} \in E(H); \text{ or }\\
 v_1 = v_2 \text{ and } \{u_1, u_2\} \in E(G)
\end{Bmatrix} \right\}.
\]

\begin{defn}
For positive integers $m,n$, the \textbf{grid graph with $m$ rows and $n$ columns}, denoted $G_{m,n}$, is the graph obtained by taking the Cartesian product $P_m \Box P_n$, where $P_i$ denotes the path graph on $i$ vertices.
\end{defn}

Now, we state a key observation about the directed $(t,r)$ broadcast domination number.

\begin{obs}\label{obs: inequality}
Given a graph $G$, let $\vec{G}$ be an arbitrary orientation of $G$. For positive integers $t$ and $r$, $\gamma_{t,r}(G) \leq \gamma_{t,r}(\vec{G})$.
\end{obs}

Observation~\ref{obs: inequality} follows from the fact that a $(t,r)$ broadcast dominating set of $\vec{G}$ is also necessarily a $(t,r)$ broadcast dominating set of $G$, as $G$ can be thought of as the result of replacing every singly-directed edge in $\vec{G}$ with a doubly-directed one, allowing reception to flow more freely within the undirected graph.
Given this observation, a natural question to ask is when this inequality becomes an equality. In other words, does there always exist an orientation $\vec{G}$ of a graph $G$ for which $\gamma_{t,r}(G) = \gamma_{t,r}(\vec{G})$?

To answer this question, we first highlight some known results for undirected grid graphs. In 2014, \cite{blessing2014tr} gave formulae for the $(t,r)$ broadcast domination numbers for $G_{m,n}$ when $m \in \{2,3,4,5\}$ for arbitrary $n \geq 0$ and small $t, r$. The authors determined these dominating patterns using an algorithm that proves that a certain pattern of broadcasting vertex placement is the optimal domination number. Some of these results are summarized below.

\begin{thm}[Theorems 2.1-2.5, \cite{blessing2014tr}]
\begin{enumerate}
    \item 
For $n \geq 3$, the $(2,2)$ broadcast domination number of the $3 \times n$ grid~is
$$
\gamma_{2,2}(G_{3 \times n}) = \ceil{\frac{4n}{3}}.
$$
\item For $n \geq 4$, the $(2,2)$ broadcast domination number of the $4 \times n$ grid~is
$$
\gamma_{2,2}(G_{4 \times n}) = 2n - \ceil{\frac{n-6}{4}}.
$$
\item For $n \geq 5$, the $(2,2)$ broadcast domination number of the $5 \times n$ grid~is
$$
\gamma_{2,2}(G_{5 \times n}) = 2n + \ceil{\frac{n+2}{7}}
$$
\item 
For $n \geq 3$, the $(3,1)$ broadcast domination number of the $3 \times n$ grid~is
$$
\gamma_{3,1}(G_{3 \times n}) = \ceil{\frac{n}{3}}.
$$
\item 
For $n \geq 4$, the $(3,1)$ broadcast domination number of the $4 \times n$ grid~is
$$
\gamma_{3,1}(G_{4 \times n}) = \floor{\frac{n+1}{7}} + \floor{\frac{n+3}{7}} + \floor{\frac{n+5}{7}} + 1.
$$
\end{enumerate}
\end{thm}

While Blessing et.\ al.\ \cite{blessing2014tr} provide formulae for $(t,r) = (3,2)$ as well, we focus on this subset of results because of our considerations in the previous section. In light of Observation~\ref{obs: inequality}, these values immediately provide lower bounds for the directed $(t,r)$ broadcast domination intervals of their corresponding directed grid graphs for the same values $t$ and $r$, assuming these graphs can each be oriented to preserve the undirected $(t,r)$ broadcast dominating sets. The following observations argue that these lower bounds are in fact tight.

\begin{obs}
In the case $(t,r) = (2,2)$, the process of directing edges is simple: make all broadcasting vertices sources, thereby making all non-broadcasting vertices sinks. 
Since each vertex in $G$ is either a broadcasting vertex or a neighbor of at least two broadcasting vertices, it suffices to observe that, for every non-broadcasting vertex $v$ in $V(G)$, its resulting in-neighborhood in the directed graph $\vec{G}$ remains its previous neighborhood in the undirected graph $G$. 
Similarly, for every broadcasting vertex, its resulting out-neighborhood is exactly its neighborhood in the undirected graph $G$. 
Thus, each vertex gets the same reception as it previously did in the undirected graph.
\end{obs}

\begin{obs}
In the case $(t,r) = (3,1)$, the intuition behind orienting the graph is to direct edges such that signal can travel as far outward from the broadcasting vertices as possible. We again begin by making all broadcasting vertices sources. Then, for each vertex which is an out-neighbor of a broadcasting vertex, direct the remainder of its not-already-directed edges outward. All other edges are incident to two non-broadcasting vertices which are each at a distance 2 away from any broadcasting vertex. This means that each of these non-broadcasting vertices receives signal 1 from that broadcasting vertex and cannot propagate signal further. As a result, this last set of edges can be oriented arbitrarily to complete the construction of a directed grid graph which maintains the $(3,1)$ dominating set. We are guaranteed that this set must in fact form a $(3,1)$ dominating set because, for any broadcasting vertex $v$, the broadcasting out-neighborhood of $v$ in the oriented graph is exactly the broadcasting neighborhood of $v$ in the undirected graph.
\end{obs}

Now, for $m \in \{2,3,4\}$ we provide intervals of integers which are properly contained in the $\dbdi{2,2}{G_{m,n}}$. 
Note that since we have shown in the previous section that the $(2,2)$ and $(3,1)$ directed broadcast domination intervals of a graph must be full, finding an orientation which achieves a domination number greater than the lower bound immediately implies the existence of orientations which achieve all values in between as well. For each of the following small grid graphs, the upper bound provided is achieved by an orientation of $G_{m,n}$ which achieves the maximum possible number of vertices with in-degree 1. Note also that the OEIS sequence \textcolor{blue}{\href{http://oeis.org/search?q=A000111&language=english&go=Search}{A000111}} counts half the number of alternating permutations on $n$ letters. We let A000111$(n)$ denote the $n$th term of this sequence. We can now state the next result.

\begin{prop} \label{prop:GridProps}
Let $n \geq 1$. Then
\begin{enumerate}
    \item $\dbdi{2,2}{G_{2,n}} \supseteq \big[n, \floor{\frac{3n + 2}{2}} \big].$
    
    \item $\dbdi{2,2}{G_{3,n}} \supseteq \big[\ceil{\frac{4n}{3}}, \floor{\frac{3n + 2}{2}} + \floor{\frac{3n + 4}{4}}\big].$
    
    \item $\dbdi{2,2}{G_{4,n}} \supseteq \big[2n - \ceil{\frac{n-6}{4}}, \floor{\frac{3n + 2}{2}} + \floor{\frac{3n + 4}{4}} + k_n\big]$, where $$
    k_n =  \floor{\text{A000111}(n+1)/\text{A000111}(n)},
    $$
    which represents the number of towers in the fourth row with $n$ columns.
\end{enumerate}
\end{prop}
\begin{proof}
As previously stated, the orientation of $G_{m,n}$ which achieves the minimal value within its respective interval is achieved by orienting the graph to preserve the dominating sets stated in \cite{blessing2014tr}. The maximal values within these given intervals are equivalent to the maximum possible number of vertices of their respective grid graphs with in-degree equal to 1. Note that any vertex with in-degree equal to 1 must be a broadcasting vertex, because it cannot receive sufficient reception from its broadcasting in-neighborhood, even if it is a direct out-neighbor of a tower. For $n = 10$, examples of orientations with maximized numbers of vertices with in-degree 1 are given in Figures \ref{fig:G2n_maximal_dominating_set_tr22}, \ref{fig:G3n_maximal_dominating_set_tr22}, and \ref{fig:G4n_maximal_dominating_set_tr22}. To verify that these values do in fact correspond to the maximum possible number of vertices with in-degree 1 can be proven by contradiction as follows. Suppose that it is possible in each of these instances to have one more vertex with in-degree 1. Then since 
$$
\sum_{v \in V(G)} deg^-(v) = |V(G)|,
$$
there must exist a vertex with an impossibly large in-degree for the given equation to still hold. 
\end{proof}

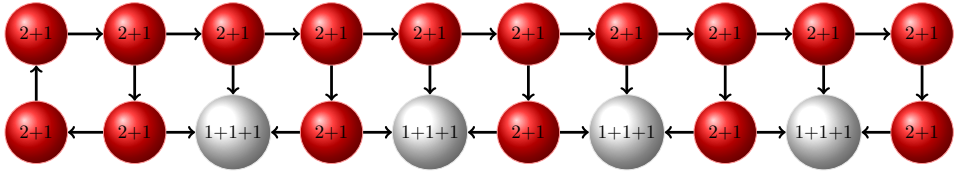
\begin{figure}[h!]
\centering
\resizebox{5in}{!}{
\begin{tikzpicture}[vertex_style/.style={circle,ball color=white,draw=black!10!white},
edge_style/.style={ thick, black, drop shadow={opacity=0.4}},vertex_style_red/.style={circle,ball color=red,draw=red!40!white},
edge_style/.style={thick, black,drop shadow={opacity=0.4}}]
\tikzstyle{ghost node}=[draw=none]
\tikzset{black node/.style={circle,draw=black, fill=white, inner sep=10.5}}
\tikzset{red node/.style={circle,draw=black, fill=red, inner sep=10.5}}

\node [vertex_style_red,minimum size=.5in] (11) at (2*1,2*1){2+1};
\node [vertex_style_red,minimum size=.5in] (12) at (2*1,2*2){2+1};
\node [vertex_style_red,minimum size=.5in] (21) at (2*2,2*1){2+1};
\node [vertex_style_red,minimum size=.5in] (22) at (2*2,2*2){2+1};
\node [vertex_style,minimum size=.5in] (31) at (2*3,2*1){1+1+1};
\node [vertex_style_red,minimum size=.5in] (32) at (2*3,2*2){2+1};
\node [vertex_style_red,minimum size=.5in] (41) at (2*4,2*1){2+1};
\node [vertex_style_red,minimum size=.5in] (42) at (2*4,2*2){2+1};
\node [vertex_style,minimum size=.5in] (51) at (2*5,2*1){1+1+1};
\node [vertex_style_red,minimum size=.5in] (52) at (2*5,2*2){2+1};
\node [vertex_style_red,minimum size=.5in] (61) at (2*6,2*1){2+1};
\node [vertex_style_red,minimum size=.5in] (62) at (2*6,2*2){2+1};
\node [vertex_style,minimum size=.5in] (71) at (2*7,2*1){1+1+1};
\node [vertex_style_red,minimum size=.5in] (72) at (2*7,2*2){2+1};
\node [vertex_style_red,minimum size=.5in] (81) at (2*8,2*1){2+1};
\node [vertex_style_red,minimum size=.5in] (82) at (2*8,2*2){2+1};
\node [vertex_style,minimum size=.5in] (91) at (2*9,2*1){1+1+1};
\node [vertex_style_red,minimum size=.5in] (92) at (2*9,2*2){2+1};
\node [vertex_style_red,minimum size=.5in] (101) at (2*10,2*1){2+1};
\node [vertex_style_red,minimum size=.5in] (102) at (2*10,2*2){2+1};

\draw[->, ultra thick] (12)--(22);
\draw[->, ultra thick] (22)--(32);
\draw[->, ultra thick] (32)--(42);
\draw[->, ultra thick] (42)--(52);
\draw[->, ultra thick] (52)--(62);
\draw[->, ultra thick] (62)--(72);
\draw[->, ultra thick] (72)--(82);
\draw[->, ultra thick] (82)--(92);
\draw[->, ultra thick] (92)--(102);

\draw[->, ultra thick] (21)--(11);
\draw[->, ultra thick] (21)--(31);
\draw[->, ultra thick] (41)--(31);
\draw[->, ultra thick] (41)--(51);
\draw[->, ultra thick] (61)--(51);
\draw[->, ultra thick] (61)--(71);
\draw[->, ultra thick] (81)--(71);
\draw[->, ultra thick] (81)--(91);
\draw[->, ultra thick] (101)--(91);

\draw[->, ultra thick] (11)--(12);
\draw[->, ultra thick] (22)--(21);
\draw[->, ultra thick] (32)--(31);
\draw[->, ultra thick] (42)--(41);
\draw[->, ultra thick] (52)--(51);
\draw[->, ultra thick] (62)--(61);
\draw[->, ultra thick] (72)--(71);
\draw[->, ultra thick] (82)--(81);
\draw[->, ultra thick] (92)--(91);
\draw[->, ultra thick] (102)--(101);

\end{tikzpicture}}
\caption{An example of the grid graph $G_{2,10}$ which achieves $\gamma_{2,2}(G_{2,n})~=~\ceil{\frac{3n+2}{2}}$}
\label{fig:G2n_maximal_dominating_set_tr22}
\end{figure}

\begin{figure}[h!]
\centering
\resizebox{5in}{!}{
\begin{tikzpicture}[vertex_style/.style={circle,ball color=white,draw=black!10!white},
edge_style/.style={ thick, black, drop shadow={opacity=0.4}},vertex_style_red/.style={circle,ball color=red,draw=red!40!white},
edge_style/.style={thick, black,drop shadow={opacity=0.4}}]
\tikzstyle{ghost node}=[draw=none]
\tikzset{black node/.style={circle,draw=black, fill=white, inner sep=10.5}}
\tikzset{red node/.style={circle,draw=black, fill=red, inner sep=10.5}}

\node [vertex_style_red,minimum size=.5in] (11) at (2*1,2*1){2+1};
\node [vertex_style_red,minimum size=.5in] (12) at (2*1,2*2){2+1};
\node [vertex_style_red,minimum size=.5in] (21) at (2*2,2*1){2+1};
\node [vertex_style_red,minimum size=.5in] (22) at (2*2,2*2){2+1};
\node [vertex_style,minimum size=.5in] (31) at (2*3,2*1){1+1+1+1};
\node [vertex_style_red,minimum size=.5in] (32) at (2*3,2*2){2+1};
\node [vertex_style_red,minimum size=.5in] (41) at (2*4,2*1){2+1};
\node [vertex_style_red,minimum size=.5in] (42) at (2*4,2*2){2+1};
\node [vertex_style,minimum size=.5in] (51) at (2*5,2*1){1+1+1+1};
\node [vertex_style_red,minimum size=.5in] (52) at (2*5,2*2){2+1};
\node [vertex_style_red,minimum size=.5in] (61) at (2*6,2*1){2+1};
\node [vertex_style_red,minimum size=.5in] (62) at (2*6,2*2){2+1};
\node [vertex_style,minimum size=.5in] (71) at (2*7,2*1){1+1+1+1};
\node [vertex_style_red,minimum size=.5in] (72) at (2*7,2*2){2+1};
\node [vertex_style_red,minimum size=.5in] (81) at (2*8,2*1){2+1};
\node [vertex_style_red,minimum size=.5in] (82) at (2*8,2*2){2+1};
\node [vertex_style,minimum size=.5in] (91) at (2*9,2*1){1+1+1+1};
\node [vertex_style_red,minimum size=.5in] (92) at (2*9,2*2){2+1};
\node [vertex_style_red,minimum size=.5in] (101) at (2*10,2*1){2+1};
\node [vertex_style_red,minimum size=.5in] (102) at (2*10,2*2){2+1};

\node [vertex_style_red,minimum size=.5in] (10) at (2*1,2*0){2+1};
\node [vertex_style,minimum size=.5in] (20) at (2*2,2*0){1+1+1};
\node [vertex_style_red,minimum size=.5in] (30) at (2*3,2*0){2+1};
\node [vertex_style_red,minimum size=.5in] (40) at (2*4,2*0){2+1};
\node [vertex_style_red,minimum size=.5in] (50) at (2*5,2*0){2+1};
\node [vertex_style,minimum size=.5in] (60) at (2*6,2*0){1+1+1};
\node [vertex_style_red,minimum size=.5in] (70) at (2*7,2*0){2+1};
\node [vertex_style_red,minimum size=.5in] (80) at (2*8,2*0){2+1};
\node [vertex_style_red,minimum size=.5in] (90) at (2*9,2*0){2+1};
\node [vertex_style,minimum size=.5in] (100) at (2*10,2*0){1+1};

\draw[->, ultra thick] (12)--(22);
\draw[->, ultra thick] (22)--(32);
\draw[->, ultra thick] (32)--(42);
\draw[->, ultra thick] (42)--(52);
\draw[->, ultra thick] (52)--(62);
\draw[->, ultra thick] (62)--(72);
\draw[->, ultra thick] (72)--(82);
\draw[->, ultra thick] (82)--(92);
\draw[->, ultra thick] (92)--(102);

\draw[->, ultra thick] (21)--(11);
\draw[->, ultra thick] (21)--(31);
\draw[->, ultra thick] (41)--(31);
\draw[->, ultra thick] (41)--(51);
\draw[->, ultra thick] (61)--(51);
\draw[->, ultra thick] (61)--(71);
\draw[->, ultra thick] (81)--(71);
\draw[->, ultra thick] (81)--(91);
\draw[->, ultra thick] (101)--(91);

\draw[->, ultra thick] (10)--(20);
\draw[->, ultra thick] (30)--(20);
\draw[->, ultra thick] (40)--(30);
\draw[->, ultra thick] (40)--(50);
\draw[->, ultra thick] (50)--(60);
\draw[->, ultra thick] (70)--(60);
\draw[->, ultra thick] (80)--(70);
\draw[->, ultra thick] (80)--(90);
\draw[->, ultra thick] (90)--(100);

\draw[->, ultra thick] (11)--(12);
\draw[->, ultra thick] (22)--(21);
\draw[->, ultra thick] (32)--(31);
\draw[->, ultra thick] (42)--(41);
\draw[->, ultra thick] (52)--(51);
\draw[->, ultra thick] (62)--(61);
\draw[->, ultra thick] (72)--(71);
\draw[->, ultra thick] (82)--(81);
\draw[->, ultra thick] (92)--(91);
\draw[->, ultra thick] (102)--(101);

\draw[->, ultra thick] (11)--(10);
\draw[->, ultra thick] (21)--(20);
\draw[->, ultra thick] (30)--(31);
\draw[->, ultra thick] (41)--(40);
\draw[->, ultra thick] (50)--(51);
\draw[->, ultra thick] (61)--(60);
\draw[->, ultra thick] (70)--(71);
\draw[->, ultra thick] (81)--(80);
\draw[->, ultra thick] (90)--(91);
\draw[->, ultra thick] (101)--(100);

\end{tikzpicture}}
\caption{An example of the grid graph $G_{3,10}$ which achieves 
$\gamma_{2,2}(G_{2,n})~=~\ceil{\frac{3n+2}{2}}~+~\floor{\frac{3n + 4}{4}}$. Note that the graph in Figure \ref{fig:G2n_maximal_dominating_set_tr22} is a subgraph.
}
\label{fig:G3n_maximal_dominating_set_tr22}
\end{figure}
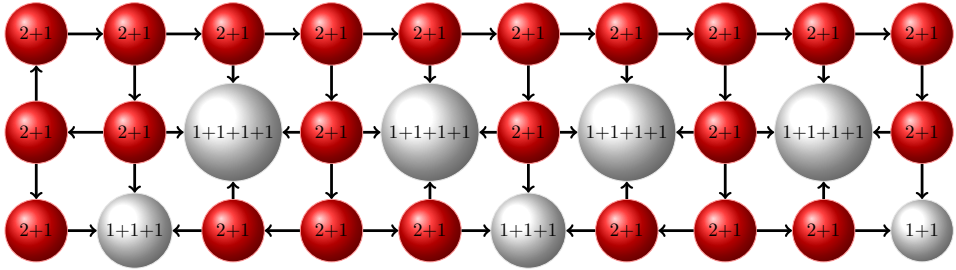

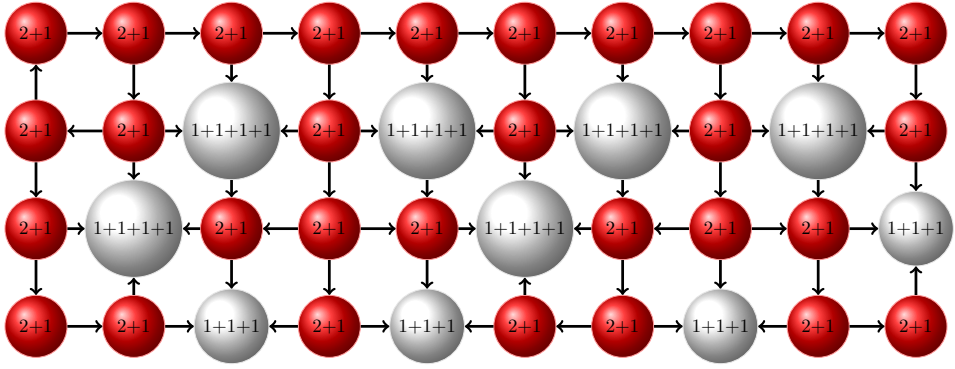
\begin{figure}[h!]
\centering
\resizebox{5in}{!}{
\begin{tikzpicture}[vertex_style/.style={circle,ball color=white,draw=black!10!white},
edge_style/.style={ thick, black, drop shadow={opacity=0.4}},vertex_style_red/.style={circle,ball color=red,draw=red!40!white},
edge_style/.style={thick, black,drop shadow={opacity=0.4}}]
\tikzstyle{ghost node}=[draw=none]
\tikzset{black node/.style={circle,draw=black, fill=white, inner sep=10.5}}
\tikzset{red node/.style={circle,draw=black, fill=red, inner sep=10.5}}

\node [vertex_style_red,minimum size=.5in] (11) at (2*1,2*1){2+1};
\node [vertex_style_red,minimum size=.5in] (12) at (2*1,2*2){2+1};
\node [vertex_style_red,minimum size=.5in] (21) at (2*2,2*1){2+1};
\node [vertex_style_red,minimum size=.5in] (22) at (2*2,2*2){2+1};
\node [vertex_style,minimum size=.5in] (31) at (2*3,2*1){1+1+1+1};
\node [vertex_style_red,minimum size=.5in] (32) at (2*3,2*2){2+1};
\node [vertex_style_red,minimum size=.5in] (41) at (2*4,2*1){2+1};
\node [vertex_style_red,minimum size=.5in] (42) at (2*4,2*2){2+1};
\node [vertex_style,minimum size=.5in] (51) at (2*5,2*1){1+1+1+1};
\node [vertex_style_red,minimum size=.5in] (52) at (2*5,2*2){2+1};
\node [vertex_style_red,minimum size=.5in] (61) at (2*6,2*1){2+1};
\node [vertex_style_red,minimum size=.5in] (62) at (2*6,2*2){2+1};
\node [vertex_style,minimum size=.5in] (71) at (2*7,2*1){1+1+1+1};
\node [vertex_style_red,minimum size=.5in] (72) at (2*7,2*2){2+1};
\node [vertex_style_red,minimum size=.5in] (81) at (2*8,2*1){2+1};
\node [vertex_style_red,minimum size=.5in] (82) at (2*8,2*2){2+1};
\node [vertex_style,minimum size=.5in] (91) at (2*9,2*1){1+1+1+1};
\node [vertex_style_red,minimum size=.5in] (92) at (2*9,2*2){2+1};
\node [vertex_style_red,minimum size=.5in] (101) at (2*10,2*1){2+1};
\node [vertex_style_red,minimum size=.5in] (102) at (2*10,2*2){2+1};

\node [vertex_style_red,minimum size=.5in] (10) at (2*1,2*0){2+1};
\node [vertex_style,minimum size=.5in] (20) at (2*2,2*0){1+1+1+1};
\node [vertex_style_red,minimum size=.5in] (30) at (2*3,2*0){2+1};
\node [vertex_style_red,minimum size=.5in] (40) at (2*4,2*0){2+1};
\node [vertex_style_red,minimum size=.5in] (50) at (2*5,2*0){2+1};
\node [vertex_style,minimum size=.5in] (60) at (2*6,2*0){1+1+1+1};
\node [vertex_style_red,minimum size=.5in] (70) at (2*7,2*0){2+1};
\node [vertex_style_red,minimum size=.5in] (80) at (2*8,2*0){2+1};
\node [vertex_style_red,minimum size=.5in] (90) at (2*9,2*0){2+1};
\node [vertex_style,minimum size=.5in] (100) at (2*10,2*0){1+1+1};

\node [vertex_style_red,minimum size=.5in] (1-1) at (2*1,2*-1){2+1};
\node [vertex_style_red,minimum size=.5in] (2-1) at (2*2,2*-1){2+1};
\node [vertex_style,minimum size=.5in] (3-1) at (2*3,2*-1){1+1+1};
\node [vertex_style_red,minimum size=.5in] (4-1) at (2*4,2*-1){2+1};
\node [vertex_style,minimum size=.5in] (5-1) at (2*5,2*-1){1+1+1};
\node [vertex_style_red,minimum size=.5in] (6-1) at (2*6,2*-1){2+1};
\node [vertex_style_red,minimum size=.5in] (7-1) at (2*7,2*-1){2+1};
\node [vertex_style,minimum size=.5in] (8-1) at (2*8,2*-1){1+1+1};
\node [vertex_style_red,minimum size=.5in] (9-1) at (2*9,2*-1){2+1};
\node [vertex_style_red,minimum size=.5in] (10-1) at (2*10,2*-1){2+1};

\draw[->, ultra thick] (12)--(22);
\draw[->, ultra thick] (22)--(32);
\draw[->, ultra thick] (32)--(42);
\draw[->, ultra thick] (42)--(52);
\draw[->, ultra thick] (52)--(62);
\draw[->, ultra thick] (62)--(72);
\draw[->, ultra thick] (72)--(82);
\draw[->, ultra thick] (82)--(92);
\draw[->, ultra thick] (92)--(102);

\draw[->, ultra thick] (21)--(11);
\draw[->, ultra thick] (21)--(31);
\draw[->, ultra thick] (41)--(31);
\draw[->, ultra thick] (41)--(51);
\draw[->, ultra thick] (61)--(51);
\draw[->, ultra thick] (61)--(71);
\draw[->, ultra thick] (81)--(71);
\draw[->, ultra thick] (81)--(91);
\draw[->, ultra thick] (101)--(91);

\draw[->, ultra thick] (10)--(20);
\draw[->, ultra thick] (30)--(20);
\draw[->, ultra thick] (40)--(30);
\draw[->, ultra thick] (40)--(50);
\draw[->, ultra thick] (50)--(60);
\draw[->, ultra thick] (70)--(60);
\draw[->, ultra thick] (80)--(70);
\draw[->, ultra thick] (80)--(90);
\draw[->, ultra thick] (90)--(100);

\draw[->, ultra thick] (1-1)--(2-1);
\draw[->, ultra thick] (2-1)--(3-1);
\draw[->, ultra thick] (4-1)--(3-1);
\draw[->, ultra thick] (4-1)--(5-1);
\draw[->, ultra thick] (6-1)--(5-1);
\draw[->, ultra thick] (7-1)--(6-1);
\draw[->, ultra thick] (7-1)--(8-1);
\draw[->, ultra thick] (9-1)--(8-1);
\draw[->, ultra thick] (9-1)--(10-1);

\draw[->, ultra thick] (11)--(12);
\draw[->, ultra thick] (22)--(21);
\draw[->, ultra thick] (32)--(31);
\draw[->, ultra thick] (42)--(41);
\draw[->, ultra thick] (52)--(51);
\draw[->, ultra thick] (62)--(61);
\draw[->, ultra thick] (72)--(71);
\draw[->, ultra thick] (82)--(81);
\draw[->, ultra thick] (92)--(91);
\draw[->, ultra thick] (102)--(101);

\draw[->, ultra thick] (11)--(10);
\draw[->, ultra thick] (21)--(20);
\draw[->, ultra thick] (31)--(30);
\draw[->, ultra thick] (41)--(40);
\draw[->, ultra thick] (51)--(50);
\draw[->, ultra thick] (61)--(60);
\draw[->, ultra thick] (71)--(70);
\draw[->, ultra thick] (81)--(80);
\draw[->, ultra thick] (91)--(90);
\draw[->, ultra thick] (101)--(100);

\draw[->, ultra thick] (10)--(1-1);
\draw[<-, ultra thick] (20)--(2-1);
\draw[->, ultra thick] (30)--(3-1);
\draw[->, ultra thick] (40)--(4-1);
\draw[->, ultra thick] (50)--(5-1);
\draw[<-, ultra thick] (60)--(6-1);
\draw[->, ultra thick] (70)--(7-1);
\draw[->, ultra thick] (80)--(8-1);
\draw[->, ultra thick] (90)--(9-1);
\draw[<-, ultra thick] (100)--(10-1);

\end{tikzpicture}}
\caption{An example of the grid graph $G_{4,10}$ which achieves 
$\gamma_{2,2}(G_{4,n})~=~\ceil{\frac{3n+2}{2}}~+~\floor{\frac{3n + 4}{4}}~+k_n$. Note that the graph in Figure \ref{fig:G3n_maximal_dominating_set_tr22} is a subgraph.
}
\label{fig:G4n_maximal_dominating_set_tr22}
\end{figure}

Unfortunately, showing set equality as opposed to the set containment proven in Proposition \ref{prop:GridProps} would require proof that all orientations of these grid graphs achieve $(t,r)$ broadcast domination interval strictly within this interval. In the case of $(2,2)$ domination, we conjecture that the proof lies in the fact that a vertex must be a dominating vertex if and only if it has in-degree at most 1, implying that the orientations provided above which maximize the number of such vertices achieve the maximal $(2,2)$ broadcast domination number. 

To aid in further understanding the directed $(t,r)$ broadcast domination intervals on small grid graphs, we provide
Sage Code \cite{Hollander2021}, which calculates the directed $(t,r)$ broadcast domination interval of an arbitrary grid graph given arbitrary positive integers $t$ and $r$. This program utilizes a dynamic programming algorithm, adapted from a similar program used in \cite{blessing2014tr}.

\subsection{Results on a Common Graph Family} \label{subsec:stars}

In the previous section, we found an interval of directed $(t,r)$ broadcast domination numbers which is contained in a (potentially larger) directed $(t,r)$ broadcast domination interval of small grid graphs. In this section, we walk through a full characterization of the directed $(t,r)$ broadcast domination interval for the star, a common and relatively simple-to-understand family of graphs. We begin by formally defining a star below.

\begin{defn}
A \textbf{star on $n$ vertices}, denoted $S_n$, consists of a single central vertex to which $n-1$ leaf vertices are adjacent.
\end{defn}

In general domination theory, the star is a simple example of a graph with domination number 1, as the central vertex is adjacent to all other vertices in the graph. Additionally, characterizing the undirected $(t,r)$ broadcast domination number of a star is also quite easy. Specifically, if $t > r$, then $\gamma_{t,r}(S_n) = 1$. Otherwise, $\gamma_{t,r}(S_n) = n$. However, once we consider all orientations of $S_n$, we see that characterizing the directed $(t,r)$ broadcast domination interval of a star becomes noticeably more involved. What follows is a full classification of $\dbdi{t,r}{S_n}$ for $n \geq 3$.

\begin{prop}\label{prop:firstStarProp}
Let $S_n$ be the star graph on $n$ vertices with $n \geq 3$, and let $t, r$ be integers such that $t = r$. If $t = r = 1$, then $\gamma_{t,r}(\vec{S_n}) = n$ for all orientations $\vec{S_n}$ of $S_n$. If $t = r = 2$, then $\dbdi{t,r}{S_n} = [n-1, n]$.
\end{prop}
\begin{proof}
The case $t = r = 1$ is trivial because all vertices must be broadcasting vertices regardless of the graph orientation. If $t = r = 2$, notice in Figure~\ref{fig:t_rquals_r_baseCase} that regardless of the orientation of the graph, all leaves must be broadcasting vertices. This is because each leaf receives reception of strength at most 1 from the central vertex, but even this reception is insufficient to dominate that leaf, so $\gamma_{2,2}(S_n) \geq n-1$ is necessary. In the case that at most 1 leaf is a source, the central vertex does not receive sufficient reception from leaves and must be a broadcasting vertex. In the case that at least 2 vertices are sources, the central vertex receives reception at least 2 and need not be a broadcasting vertex. Thus $\dbdi{2,2}{S_n} = [n-1, n]$.
\end{proof}

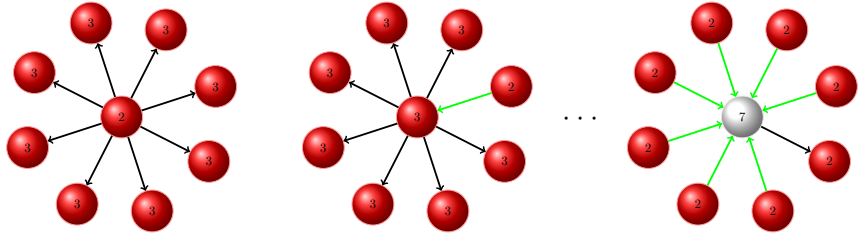
\begin{figure}[h!]
\centering
\resizebox{1.25in}{!}{
\begin{tikzpicture}
[vertex_style/.style={circle,ball color=white,draw=black!10!white},edge_style/.style={ thick, black},vertex_style_red/.style={circle,ball color=red,draw=red!40!white},vertex_style_blue/.style={circle,ball color=blue,draw=blue!40!white},edge_style/.style={thick, black},rotate=18]
\node[vertex_style_red,minimum size=.5in](a0) at (0,0) {2};
\node[vertex_style_red,minimum size=.5in](a1) at (3,0) {3};
\node[vertex_style_red,minimum size=.5in](a2) at (2.1,2.1) {3};
\node[vertex_style_red,minimum size=.5in](a3) at (0,3) {3};
\node[vertex_style_red,minimum size=.5in](a4) at (-2.1,2.1) {3};
\node[vertex_style_red,minimum size=.5in](a5) at (-3,0) {3};
\node[vertex_style_red,minimum size=.5in](a6) at (-2.1,-2.1) {3};
\node[vertex_style_red,minimum size=.5in](a7) at (0,-3) {3};
\node[vertex_style_red,minimum size=.5in](a8) at (2.1,-2.1) {3};
\draw[->, ultra thick] (a0)--(a1);
\draw[->, ultra thick] (a0)--(a2);
\draw[->, ultra thick] (a0)--(a3);
\draw[->, ultra thick] (a0)--(a4);
\draw[->, ultra thick] (a0)--(a5);
\draw[->, ultra thick] (a0)--(a6);
\draw[->, ultra thick] (a0)--(a7);
\draw[->, ultra thick] (a0)--(a8);
\end{tikzpicture}
}
\hspace{.5cm}
\resizebox{1.25in}{!}{
\begin{tikzpicture}
[vertex_style/.style={circle,ball color=white,draw=black!10!white},edge_style/.style={ thick, black,drop shadow={opacity=0.4}},vertex_style_red/.style={circle,ball color=red,draw=red!40!white},vertex_style_blue/.style={circle,ball color=blue,draw=blue!40!white,drop shadow={opacity=0.2}},edge_style/.style={thick, black,drop shadow={opacity=0.4}},rotate=18]
\node[vertex_style_red,minimum size=.5in](a0) at (0,0) {3};
\node[vertex_style_red,minimum size=.5in](a1) at (3,0) {2};
\node[vertex_style_red,minimum size=.5in](a2) at (2.1,2.1) {3};
\node[vertex_style_red,minimum size=.5in](a3) at (0,3) {3};
\node[vertex_style_red,minimum size=.5in](a4) at (-2.1,2.1) {3};
\node[vertex_style_red,minimum size=.5in](a5) at (-3,0) {3};
\node[vertex_style_red,minimum size=.5in](a6) at (-2.1,-2.1) {3};
\node[vertex_style_red,minimum size=.5in](a7) at (0,-3) {3};
\node[vertex_style_red,minimum size=.5in](a8) at (2.1,-2.1) {3};
\draw[<-, ultra thick, color = green] (a0)--(a1);
\draw[->, ultra thick] (a0)--(a2);
\draw[->, ultra thick] (a0)--(a3);
\draw[->, ultra thick] (a0)--(a4);
\draw[->, ultra thick] (a0)--(a5);
\draw[->, ultra thick] (a0)--(a6);
\draw[->, ultra thick] (a0)--(a7);
\draw[->, ultra thick] (a0)--(a8);
\end{tikzpicture}
}
\includegraphics[width=.35in]{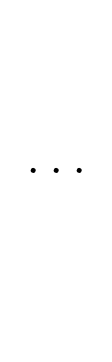}
\resizebox{1.25in}{!}{
\begin{tikzpicture}
[vertex_style/.style={circle,ball color=white,draw=black!10!white},edge_style/.style={ thick, black,drop shadow={opacity=0.4}},vertex_style_red/.style={circle,ball color=red,draw=red!40!white},vertex_style_blue/.style={circle,ball color=blue,draw=blue!40!white,drop shadow={opacity=0.2}},edge_style/.style={thick, black,drop shadow={opacity=0.4}},rotate=18]
\node[vertex_style,minimum size=.5in](a0) at (0,0) {7};
\node[vertex_style_red,minimum size=.5in](a1) at (3,0) {2};
\node[vertex_style_red,minimum size=.5in](a2) at (2.1,2.1) {2};
\node[vertex_style_red,minimum size=.5in](a3) at (0,3) {2};
\node[vertex_style_red,minimum size=.5in](a4) at (-2.1,2.1) {2};
\node[vertex_style_red,minimum size=.5in](a5) at (-3,0) {2};
\node[vertex_style_red,minimum size=.5in](a6) at (-2.1,-2.1) {2};
\node[vertex_style_red,minimum size=.5in](a7) at (0,-3) {2};
\node[vertex_style_red,minimum size=.5in](a8) at (2.1,-2.1) {2};
\draw[<-, ultra thick, color = green] (a0)--(a1);
\draw[<-, ultra thick, color = green] (a0)--(a2);
\draw[<-, ultra thick, color = green] (a0)--(a3);
\draw[<-, ultra thick, color = green] (a0)--(a4);
\draw[<-, ultra thick, color = green] (a0)--(a5);
\draw[<-, ultra thick, color = green] (a0)--(a6);
\draw[<-, ultra thick, color = green] (a0)--(a7);
\draw[->, ultra thick] (a0)--(a8);
\end{tikzpicture}
 }
\caption{If fewer than two leaves are sources, the center must be a broadcasting vertex. Flipped arcs are shown in green, and broadcasting vertices are highlighted in red.}
\label{fig:t_rquals_r_baseCase}
\end{figure}

\begin{prop}\label{prop:secondStarProp}
Let $S_n$ be the star graph on $n$ vertices with $n \geq 3$, and let $t, r$ be integers. If $t = r > 2$, then $\dbdi{t,r}{S_n} = [2,n]$.
\end{prop}
\begin{proof}
Let $s$ denote the number of source leaves in an orientation of $S_n$, and let $S_n^s$ denote the orientation of the star with $s$ source leaves. To show that $\dbdi{t,r}{S_n} = [2,n]$, we start with $S_n^0$, the leftmost graph in Figure~\ref{fig:t_rquals_r_otherCaseA}.  The central vertex is required to be a dominating vertex because it has in-degree zero, and all leaves receive insufficient reception from this central vertex, requiring them each to be dominating vertices as well and making the domination number in this orientation $n$. 
Flipping one arc, we move to the case $S_n^1$, where now the single source leaf and the central vertex form the dominating set because each sink leaf receives reception $(t-1) + (t-2) \geq r$. Both of these orientations are shown in Figure~\ref{fig:t_rquals_r_baseCase}.
We now proceed to flip the remaining arcs, one at a time, so that $s$ increases from $1$ to $n-1$. Each of these flips increases $\gamma_{t,r}$ by 1 because a leaf which has previously been a sink is now a source, and therefore it receives insufficient reception and must be added to the set of broadcasting vertices. The initial and final such orientations are shown in Figure~\ref{fig:t_rquals_r_otherCaseB}. 
Since this construction iterates over all possible orientations of $S_n$, we conclude by exhaustion that $\dbdi{t,r}{S_n} = [2,n]$.
\end{proof}

\begin{figure}[h!]
\centering
\resizebox{1.25in}{!}{
\begin{tikzpicture}
[vertex_style/.style={circle,ball color=white,draw=black!10!white},edge_style/.style={ thick, black,drop shadow={opacity=0.4}},vertex_style_red/.style={circle,ball color=red,draw=red!40!white},vertex_style_blue/.style={circle,ball color=blue,draw=blue!40!white,drop shadow={opacity=0.2}},edge_style/.style={thick, black,drop shadow={opacity=0.4}},rotate=18]
\node[vertex_style_red,minimum size=.6in](a0) at (0,0) {$t$};
\node[vertex_style_red,minimum size=.5in](a1) at (3,0) {$2t-1$};
\node[vertex_style_red,minimum size=.5in](a2) at (2.1,2.1) {$2t-1$};
\node[vertex_style_red,minimum size=.5in](a3) at (0,3) {$2t-1$};
\node[vertex_style_red,minimum size=.5in](a4) at (-2.1,2.1) {$2t-1$};
\node[vertex_style_red,minimum size=.5in](a5) at (-3,0) {$2t-1$};
\node[vertex_style_red,minimum size=.5in](a6) at (-2.1,-2.1) {$2t-1$};
\node[vertex_style_red,minimum size=.5in](a7) at (0,-3) {$2t-1$};
\node[vertex_style_red,minimum size=.5in](a8) at (2.1,-2.1) {$2t-1$};
\draw[->, ultra thick] (a0)--(a1);
\draw[->, ultra thick] (a0)--(a2);
\draw[->, ultra thick] (a0)--(a3);
\draw[->, ultra thick] (a0)--(a4);
\draw[->, ultra thick] (a0)--(a5);
\draw[->, ultra thick] (a0)--(a6);
\draw[->, ultra thick] (a0)--(a7);
\draw[->, ultra thick] (a0)--(a8);
\end{tikzpicture}
}
\hspace{1cm}
\resizebox{1.25in}{!}{
\begin{tikzpicture}
[vertex_style/.style={circle,ball color=white,draw=black!10!white},edge_style/.style={ thick, black,drop shadow={opacity=0.4}},vertex_style_red/.style={circle,ball color=red,draw=red!40!white},vertex_style_blue/.style={circle,ball color=blue,draw=blue!40!white,drop shadow={opacity=0.2}},edge_style/.style={thick, black,drop shadow={opacity=0.4}},rotate=18]
\node[vertex_style_red,minimum size=.6in](a0) at (0,0) {$2t-1$};
\node[vertex_style_red,minimum size=.6in](a1) at (3,0) {$t$};
\node[vertex_style,minimum size=.5in](a2) at (2.1,2.1) {$2t-3$};
\node[vertex_style,minimum size=.5in](a3) at (0,3) {$2t-3$};
\node[vertex_style,minimum size=.5in](a4) at (-2.1,2.1) {$2t-3$};
\node[vertex_style,minimum size=.5in](a5) at (-3,0) {$2t-3$};
\node[vertex_style,minimum size=.5in](a6) at (-2.1,-2.1) {$2t-3$};
\node[vertex_style,minimum size=.5in](a7) at (0,-3) {$2t-3$};
\node[vertex_style,minimum size=.5in](a8) at (2.1,-2.1) {$2t-3$};
\draw[<-, ultra thick, color = green] (a0)--(a1);
\draw[->, ultra thick] (a0)--(a2);
\draw[->, ultra thick] (a0)--(a3);
\draw[->, ultra thick] (a0)--(a4);
\draw[->, ultra thick] (a0)--(a5);
\draw[->, ultra thick] (a0)--(a6);
\draw[->, ultra thick] (a0)--(a7);
\draw[->, ultra thick] (a0)--(a8);
\end{tikzpicture}
}
\caption{The cases $S_n^0$ and $S_n^1$ are shown. Flipped arcs are shown in green, and broadcasting vertices are highlighted in red.}
\label{fig:t_rquals_r_otherCaseA}
\end{figure}
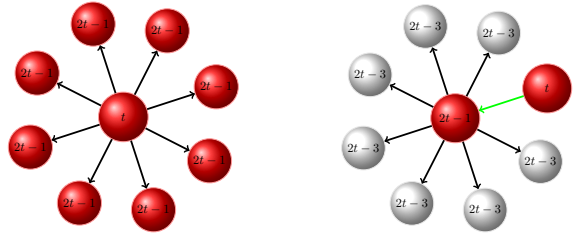

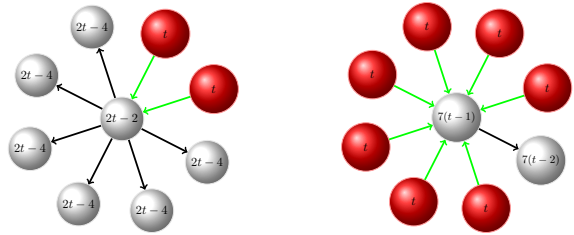
\begin{figure}[h!]
\centering
\resizebox{1.25in}{!}{
\begin{tikzpicture}
[vertex_style/.style={circle,ball color=white,draw=black!10!white},edge_style/.style={ thick, black,drop shadow={opacity=0.4}},vertex_style_red/.style={circle,ball color=red,draw=red!40!white},vertex_style_blue/.style={circle,ball color=blue,draw=blue!40!white,drop shadow={opacity=0.2}},edge_style/.style={thick, black,drop shadow={opacity=0.4}},rotate=18]
\node[vertex_style,minimum size=.5in](a0) at (0,0) {$2t-2$};
\node[vertex_style_red,minimum size=.6in](a1) at (3,0) {$t$};
\node[vertex_style_red,minimum size=.6in](a2) at (2.1,2.1) {$t$};
\node[vertex_style,minimum size=.5in](a3) at (0,3) {$2t-4$};
\node[vertex_style,minimum size=.5in](a4) at (-2.1,2.1) {$2t-4$};
\node[vertex_style,minimum size=.5in](a5) at (-3,0) {$2t-4$};
\node[vertex_style,minimum size=.5in](a6) at (-2.1,-2.1) {$2t-4$};
\node[vertex_style,minimum size=.5in](a7) at (0,-3) {$2t-4$};
\node[vertex_style,minimum size=.5in](a8) at (2.1,-2.1) {$2t-4$};
\draw[<-, ultra thick, color = green] (a0)--(a1);
\draw[<-, ultra thick, color = green] (a0)--(a2);
\draw[->, ultra thick] (a0)--(a3);
\draw[->, ultra thick] (a0)--(a4);
\draw[->, ultra thick] (a0)--(a5);
\draw[->, ultra thick] (a0)--(a6);
\draw[->, ultra thick] (a0)--(a7);
\draw[->, ultra thick] (a0)--(a8);
\end{tikzpicture}
}
\hspace{1cm}
\resizebox{1.25in}{!}{
\begin{tikzpicture}
[vertex_style/.style={circle,ball color=white,draw=black!10!white},edge_style/.style={ thick, black,drop shadow={opacity=0.4}},vertex_style_red/.style={circle,ball color=red,draw=red!40!white},vertex_style_blue/.style={circle,ball color=blue,draw=blue!40!white,drop shadow={opacity=0.2}},edge_style/.style={thick, black,drop shadow={opacity=0.4}},rotate=18]
\node[vertex_style,minimum size=.5in](a0) at (0,0) {\small$7(t-1)$};
\node[vertex_style_red,minimum size=.6in](a1) at (3,0) {$t$};
\node[vertex_style_red,minimum size=.6in](a2) at (2.1,2.1) {$t$};
\node[vertex_style_red,minimum size=.6in](a3) at (0,3) {$t$};
\node[vertex_style_red,minimum size=.6in](a4) at (-2.1,2.1) {$t$};
\node[vertex_style_red,minimum size=.6in](a5) at (-3,0) {$t$};
\node[vertex_style_red,minimum size=.6in](a6) at (-2.1,-2.1) {$t$};
\node[vertex_style_red,minimum size=.6in](a7) at (0,-3) {$t$};
\node[vertex_style,minimum size=.5in](a8) at (2.1,-2.1) {\small$7(t-2)$};
\draw[<-, ultra thick, color = green] (a0)--(a1);
\draw[<-, ultra thick, color = green] (a0)--(a2);
\draw[<-, ultra thick, color = green] (a0)--(a3);
\draw[<-, ultra thick, color = green] (a0)--(a4);
\draw[<-, ultra thick, color = green] (a0)--(a5);
\draw[<-, ultra thick, color = green] (a0)--(a6);
\draw[<-, ultra thick, color = green] (a0)--(a7);
\draw[->, ultra thick] (a0)--(a8);
\end{tikzpicture}
}
\caption{The cases $S_n^2$ and $S_n^{n-2}$ are shown. Flipped arcs are shown in green, and broadcasting vertices are highlighted in red.}
\label{fig:t_rquals_r_otherCaseB}
\end{figure}

\begin{prop} \label{prop:thirdStarProp}
Let $S_n$ be the star graph on $n$ vertices with $n \geq 3$, and let $t, r$ be integers. If $t > r$, then $\dbdi{t,r}{S_n} = [1,n-1]$.
\end{prop}
\begin{proof}
Once again, let $s$ denote the number of source leaves in an orientation of $S_n$, and let $S_n^s$ denote the orientation of the star with $s$ source leaves. To show that $\dbdi{t,r}{S_n} = [1,n-1]$, we start with $S_n^0$.  The central vertex is required to be a dominating vertex because it has in-degree zero, and all leaves receive sufficient reception from this central vertex, making the domination number of this orientation 1.
We now proceed to flip the orientation of each edge so that $s$ increases from $0$ to $n-1$. 
If $(t,r) = (2,1)$, each of these flips increases the domination number by 1 except for the last flip, when $s$ increases from $n-2$ to $n-1$. This is because the set of broadcasting vertices gains the last leaf vertex but loses the central vertex when this arc is flipped. This case is highlighted in Figure~\ref{fig:t_greaterThan_r_A}.
For all other values of $(t,r)$, meaning all instances when $t - r \geq 2$, each of these arc flips increases the domination number of the resulting graph by 1 except for when $s$ increases from 1 to 2. This is true because when $s \geq 2$, the central vertex no longer needs to be a dominating vertex, as all vertices receive reception at least $2(t-2)$ from the source leaves. Note that, given the constraints, $2(t-2) \geq r$ for all $t \geq 3, t > r$, so this is indeed sufficient reception. This case is highlighted in Figures~\ref{fig:t_greaterThan_r_B_pt1} and~\ref{fig:t_greaterThan_r_B_pt2}.
In either case, we get that $\dbdi{t,r}{S_n} = [1,n-1]$.
\end{proof}

\begin{figure}[h!]
\centering
\resizebox{1.25in}{!}{
\begin{tikzpicture}
[vertex_style/.style={circle,ball color=white,draw=black!10!white},edge_style/.style={ thick, black,drop shadow={opacity=0.4}},vertex_style_red/.style={circle,ball color=red,draw=red!40!white},vertex_style_blue/.style={circle,ball color=blue,draw=blue!40!white,drop shadow={opacity=0.2}},edge_style/.style={thick, black,drop shadow={opacity=0.4}},rotate=18]
\node[vertex_style_red,minimum size=.5in](a0) at (0,0) {2};
\node[vertex_style,minimum size=.5in](a1) at (3,0) {1};
\node[vertex_style,minimum size=.5in](a2) at (2.1,2.1) {1};
\node[vertex_style,minimum size=.5in](a3) at (0,3) {1};
\node[vertex_style,minimum size=.5in](a4) at (-2.1,2.1) {1};
\node[vertex_style,minimum size=.5in](a5) at (-3,0) {1};
\node[vertex_style,minimum size=.5in](a6) at (-2.1,-2.1) {1};
\node[vertex_style,minimum size=.5in](a7) at (0,-3) {1};
\node[vertex_style,minimum size=.5in](a8) at (2.1,-2.1) {1};
\draw[->, ultra thick] (a0)--(a1);
\draw[->, ultra thick] (a0)--(a2);
\draw[->, ultra thick] (a0)--(a3);
\draw[->, ultra thick] (a0)--(a4);
\draw[->, ultra thick] (a0)--(a5);
\draw[->, ultra thick] (a0)--(a6);
\draw[->, ultra thick] (a0)--(a7);
\draw[->, ultra thick] (a0)--(a8);
\end{tikzpicture}
}
\includegraphics[width=.35in]{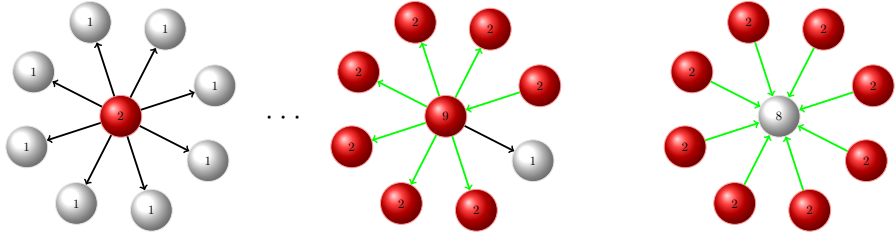}
\resizebox{1.25in}{!}{
\begin{tikzpicture}
[vertex_style/.style={circle,ball color=white,draw=black!10!white},edge_style/.style={ thick, black,drop shadow={opacity=0.4}},vertex_style_red/.style={circle,ball color=red,draw=red!40!white},vertex_style_blue/.style={circle,ball color=blue,draw=blue!40!white,drop shadow={opacity=0.2}},edge_style/.style={thick, black,drop shadow={opacity=0.4}},rotate=18]
\node[vertex_style_red,minimum size=.5in](a0) at (0,0) {9};
\node[vertex_style_red,minimum size=.5in](a1) at (3,0) {2};
\node[vertex_style_red,minimum size=.5in](a2) at (2.1,2.1) {2};
\node[vertex_style_red,minimum size=.5in](a3) at (0,3) {2};
\node[vertex_style_red,minimum size=.5in](a4) at (-2.1,2.1) {2};
\node[vertex_style_red,minimum size=.5in](a5) at (-3,0) {2};
\node[vertex_style_red,minimum size=.5in](a6) at (-2.1,-2.1) {2};
\node[vertex_style_red,minimum size=.5in](a7) at (0,-3) {2};
\node[vertex_style,minimum size=.5in](a8) at (2.1,-2.1) {1};
\draw[<-, ultra thick, color = green] (a0)--(a1);
\draw[->, ultra thick, color = green] (a0)--(a2);
\draw[->, ultra thick, color = green] (a0)--(a3);
\draw[->, ultra thick, color = green] (a0)--(a4);
\draw[->, ultra thick, color = green] (a0)--(a5);
\draw[->, ultra thick, color = green] (a0)--(a6);
\draw[->, ultra thick, color = green] (a0)--(a7);
\draw[->, ultra thick] (a0)--(a8);
\end{tikzpicture}
}
\hspace{1cm}
\resizebox{1.25in}{!}{
\begin{tikzpicture}
[vertex_style/.style={circle,ball color=white,draw=black!10!white},edge_style/.style={ thick, black,drop shadow={opacity=0.4}},vertex_style_red/.style={circle,ball color=red,draw=red!40!white},vertex_style_blue/.style={circle,ball color=blue,draw=blue!40!white,drop shadow={opacity=0.2}},edge_style/.style={thick, black,drop shadow={opacity=0.4}},rotate=18]
\node[vertex_style,minimum size=.5in](a0) at (0,0) {8};
\node[vertex_style_red,minimum size=.5in](a1) at (3,0) {2};
\node[vertex_style_red,minimum size=.5in](a2) at (2.1,2.1) {2};
\node[vertex_style_red,minimum size=.5in](a3) at (0,3) {2};
\node[vertex_style_red,minimum size=.5in](a4) at (-2.1,2.1) {2};
\node[vertex_style_red,minimum size=.5in](a5) at (-3,0) {2};
\node[vertex_style_red,minimum size=.5in](a6) at (-2.1,-2.1) {2};
\node[vertex_style_red,minimum size=.5in](a7) at (0,-3) {2};
\node[vertex_style_red,minimum size=.5in](a8) at (2.1,-2.1) {2};
\draw[<-, ultra thick, color = green] (a0)--(a1);
\draw[<-, ultra thick, color = green] (a0)--(a2);
\draw[<-, ultra thick, color = green] (a0)--(a3);
\draw[<-, ultra thick, color = green] (a0)--(a4);
\draw[<-, ultra thick, color = green] (a0)--(a5);
\draw[<-, ultra thick, color = green] (a0)--(a6);
\draw[<-, ultra thick, color = green] (a0)--(a7);
\draw[<-, ultra thick, color = green] (a0)--(a8);
\end{tikzpicture}
 }
\caption{The case $(t,r) = (2,1)$. Green arcs denote arcs which have been flipped from the original orientation $S_n^0$.}
\label{fig:t_greaterThan_r_A}
\end{figure}

\begin{figure}[h!]
\centering
\resizebox{1.25in}{!}{
\begin{tikzpicture}
[vertex_style/.style={circle,ball color=white,draw=black!10!white},edge_style/.style={ thick, black,drop shadow={opacity=0.4}},vertex_style_red/.style={circle,ball color=red,draw=red!40!white},vertex_style_blue/.style={circle,ball color=blue,draw=blue!40!white,drop shadow={opacity=0.2}},edge_style/.style={thick, black,drop shadow={opacity=0.4}},rotate=18]
\node[vertex_style_red,minimum size=.6in](a0) at (0,0) {$t$};
\node[vertex_style,minimum size=.6in](a1) at (3,0) {$t-1$};
\node[vertex_style,minimum size=.6in](a2) at (2.1,2.1) {$t-1$};
\node[vertex_style,minimum size=.6in](a3) at (0,3) {$t-1$};
\node[vertex_style,minimum size=.6in](a4) at (-2.1,2.1) {$t-1$};
\node[vertex_style,minimum size=.6in](a5) at (-3,0) {$t-1$};
\node[vertex_style,minimum size=.6in](a6) at (-2.1,-2.1) {$t-1$};
\node[vertex_style,minimum size=.6in](a7) at (0,-3) {$t-1$};
\node[vertex_style,minimum size=.6in](a8) at (2.1,-2.1) {$t-1$};
\draw[->, ultra thick] (a0)--(a1);
\draw[->, ultra thick] (a0)--(a2);
\draw[->, ultra thick] (a0)--(a3);
\draw[->, ultra thick] (a0)--(a4);
\draw[->, ultra thick] (a0)--(a5);
\draw[->, ultra thick] (a0)--(a6);
\draw[->, ultra thick] (a0)--(a7);
\draw[->, ultra thick] (a0)--(a8);
\end{tikzpicture}
}
\hspace{1cm}
\resizebox{1.25in}{!}{
\begin{tikzpicture}
[vertex_style/.style={circle,ball color=white,draw=black!10!white},edge_style/.style={ thick, black,drop shadow={opacity=0.4}},vertex_style_red/.style={circle,ball color=red,draw=red!40!white},vertex_style_blue/.style={circle,ball color=blue,draw=blue!40!white,drop shadow={opacity=0.2}},edge_style/.style={thick, black,drop shadow={opacity=0.4}},rotate=18]
\node[vertex_style_red,minimum size=.6in](a0) at (0,0) {$2t-1$};
\node[vertex_style_red,minimum size=.6in](a1) at (3,0) {$t$};
\node[vertex_style,minimum size=.6in](a2) at (2.1,2.1) {$2t-3$};
\node[vertex_style,minimum size=.6in](a3) at (0,3) {$2t-3$};
\node[vertex_style,minimum size=.6in](a4) at (-2.1,2.1) {$2t-3$};
\node[vertex_style,minimum size=.6in](a5) at (-3,0) {$2t-3$};
\node[vertex_style,minimum size=.6in](a6) at (-2.1,-2.1) {$2t-3$};
\node[vertex_style,minimum size=.6in](a7) at (0,-3) {$2t-3$};
\node[vertex_style,minimum size=.6in](a8) at (2.1,-2.1) {$2t-3$};
\draw[<-, ultra thick, color = green] (a0)--(a1);
\draw[->, ultra thick] (a0)--(a2);
\draw[->, ultra thick] (a0)--(a3);
\draw[->, ultra thick] (a0)--(a4);
\draw[->, ultra thick] (a0)--(a5);
\draw[->, ultra thick] (a0)--(a6);
\draw[->, ultra thick] (a0)--(a7);
\draw[->, ultra thick] (a0)--(a8);
\end{tikzpicture}
}
\caption{Orientations $S_n^0$ and $S_n^1$ with $t > r$. Green arcs denote arcs which have been flipped from the original orientation $S_n^0$, and broadcasting vertices are highlighted in red.}
\label{fig:t_greaterThan_r_B_pt1}
\end{figure}
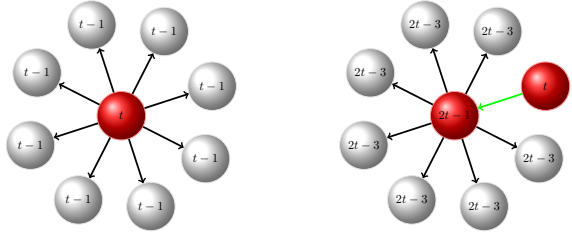

\begin{figure}[h!]
\centering
\resizebox{1.25in}{!}{
\begin{tikzpicture}
[vertex_style/.style={circle,ball color=white,draw=black!10!white},edge_style/.style={ thick, black,drop shadow={opacity=0.4}},vertex_style_red/.style={circle,ball color=red,draw=red!40!white},vertex_style_blue/.style={circle,ball color=blue,draw=blue!40!white,drop shadow={opacity=0.2}},edge_style/.style={thick, black,drop shadow={opacity=0.4}},rotate=18]
\node[vertex_style,minimum size=.6in](a0) at (0,0) {$2t-2$};
\node[vertex_style_red,minimum size=.6in](a1) at (3,0) {$t$};
\node[vertex_style_red,minimum size=.6in](a2) at (2.1,2.1) {$t$};
\node[vertex_style,minimum size=.6in](a3) at (0,3) {$2t-4$};
\node[vertex_style,minimum size=.6in](a4) at (-2.1,2.1) {$2t-4$};
\node[vertex_style,minimum size=.6in](a5) at (-3,0) {$2t-4$};
\node[vertex_style,minimum size=.6in](a6) at (-2.1,-2.1) {$2t-4$};
\node[vertex_style,minimum size=.6in](a7) at (0,-3) {$2t-4$};
\node[vertex_style,minimum size=.6in](a8) at (2.1,-2.1) {$2t-4$};
\draw[<-, ultra thick, color = green] (a0)--(a1);
\draw[<-, ultra thick, color = green] (a0)--(a2);
\draw[->, ultra thick] (a0)--(a3);
\draw[->, ultra thick] (a0)--(a4);
\draw[->, ultra thick] (a0)--(a5);
\draw[->, ultra thick] (a0)--(a6);
\draw[->, ultra thick] (a0)--(a7);
\draw[->, ultra thick] (a0)--(a8);
\end{tikzpicture}}
\includegraphics[width=.3in]{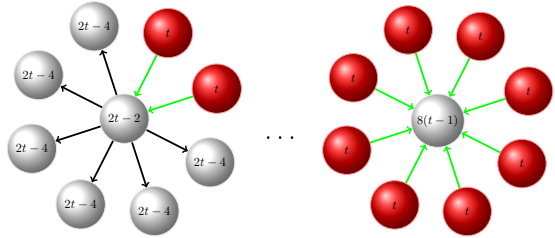}
\resizebox{1.25in}{!}{
\begin{tikzpicture}
[vertex_style/.style={circle,ball color=white,draw=black!10!white},edge_style/.style={ thick, black,drop shadow={opacity=0.4}},vertex_style_red/.style={circle,ball color=red,draw=red!40!white},vertex_style_blue/.style={circle,ball color=blue,draw=blue!40!white,drop shadow={opacity=0.2}},edge_style/.style={thick, black,drop shadow={opacity=0.4}},rotate=18]
\node[vertex_style,minimum size=.6in](a0) at (0,0) {$8(t-1)$};
\node[vertex_style_red,minimum size=.6in](a1) at (3,0) {$t$};
\node[vertex_style_red,minimum size=.6in](a2) at (2.1,2.1) {$t$};
\node[vertex_style_red,minimum size=.6in](a3) at (0,3) {$t$};
\node[vertex_style_red,minimum size=.6in](a4) at (-2.1,2.1) {$t$};
\node[vertex_style_red,minimum size=.6in](a5) at (-3,0) {$t$};
\node[vertex_style_red,minimum size=.6in](a6) at (-2.1,-2.1) {$t$};
\node[vertex_style_red,minimum size=.6in](a7) at (0,-3) {$t$};
\node[vertex_style_red,minimum size=.6in](a8) at (2.1,-2.1) {$t$};
\draw[<-, ultra thick, color = green] (a0)--(a1);
\draw[<-, ultra thick, color = green] (a0)--(a2);
\draw[<-, ultra thick, color = green] (a0)--(a3);
\draw[<-, ultra thick, color = green] (a0)--(a4);
\draw[<-, ultra thick, color = green] (a0)--(a5);
\draw[<-, ultra thick, color = green] (a0)--(a6);
\draw[<-, ultra thick, color = green] (a0)--(a7);
\draw[<-, ultra thick, color = green] (a0)--(a8);
\end{tikzpicture}
}
\caption{All other orientations with $t > r$. Purple arcs denote arcs which have been flipped from the original orientation $S_n^0$, and broadcasting vertices are highlighted in red.}
\label{fig:t_greaterThan_r_B_pt2}
\end{figure}

We remark that this result leads to a noteworthy result about the $\dbdi{t,r}{S_n}$.

\begin{thm}\label{thm:starFullness}
Let $t,r,d,D$ be positive integers such that $[d,D] = \dbdi{t,r}{S_n}$ and $t \geq r$. Then, for every $x \in [d,D]$, there exists an orientation $\vec{S}_{n_x}$ such that $\gamma_{t,r}(\vec{S}_{n_x}) = x$. In other words, the directed $(t,r)$ broadcast domination interval of a star on $n$ vertices is always full.
\end{thm}

\begin{proof}
By exhaustion, using the previous propositions within this section.
\end{proof}

As exhibited by the star graph, finding the directed $(t,r)$ broadcast domination interval can be very difficult, even for simple families of graphs. Moreover, even if finding a dominating ``strategy" may be straightforward, proving that the interval resulting from that ``strategy" is equal to that graph's $(t,r)$ directed broadcast domination interval can be quite difficult, as exhibited by our findings from this section and from Section \ref{subsec:small_grid}.

\section{On Directed \texorpdfstring{$(t,r)$}{(t,r)} Broadcast Domination of the Infinite Grid} \label{chap:infinite_grid}

We now shift our discussion to $(t,r)$ broadcast domination on the infinite grid graph, denoted by $\mathbb{Z}\times\mathbb{Z}$. In previous literature, Blessing et al.\ and  Harris et al.\ discuss efficient domination of the infinite Cartesian and triangular lattices, respectively \cite{blessing2014tr,harris2018broadcast}. 
In this section, we introduce and provide some initial results when considering the extension of their work on efficient $(t,r)$ broadcast domination of the infinite Cartesian lattice to the directed variant.
We remark that because the infinite grid cannot have a finite domination number, efficiency is instead measured as follows. 
Intuitively, an efficient broadcast dominating set is one which wastes the least amount of signal possible.
We give a rigorous definition of this idea below, and we note that this definition holds for any orientation $\vec{G}$ of $G$.

\begin{defn} (\cite{harris2018broadcast})
A $(t,r)$ broadcast dominating set $S$ for $G$ is said to be \textbf{efficient} if for all $u \in V(G)$,
$$
r(u) = \begin{cases} 
r &\mbox{if } d(u,v) \geq t - r \text{ for all }v \in S \\
r - d(u,v) & \mbox{if } 0\leq d(u,v) < t - r \text{ for exactly one }v \in S.
\end{cases}
$$
\end{defn}

In other words, non-broadcasting vertices far away from broadcasting vertices receive only the minimum required signal and do not `waste' signal by receiving more than they need. Of course, vertices which are relatively close to broadcasting vertices are not penalized for having more than the minimum required reception because they must continue transmitting signal to vertices which are farther away.

Now equipped with a notion of efficiency, we present our main result.

\begin{thm} \label{thm:one_third_two_third}
There exist orientations of the graph $\mathbb{Z}\times\mathbb{Z}$ that achieve efficient directed $(2,2)$ broadcast domination densities of $\frac13$, $\frac12$, and $\frac23$.
\end{thm}

\begin{proof}
The proof is via construction. To give each of these results, consider the orientations provided in Figure~\ref{img:infinite,density1/3:2/3} and note that the number of vertices used in every horizontal line are $\frac13$ and $\frac23$ of the total number of vertices, respectively.
\end{proof}

\begin{figure}[h!]
    \centering
    \includegraphics[width = \textwidth]{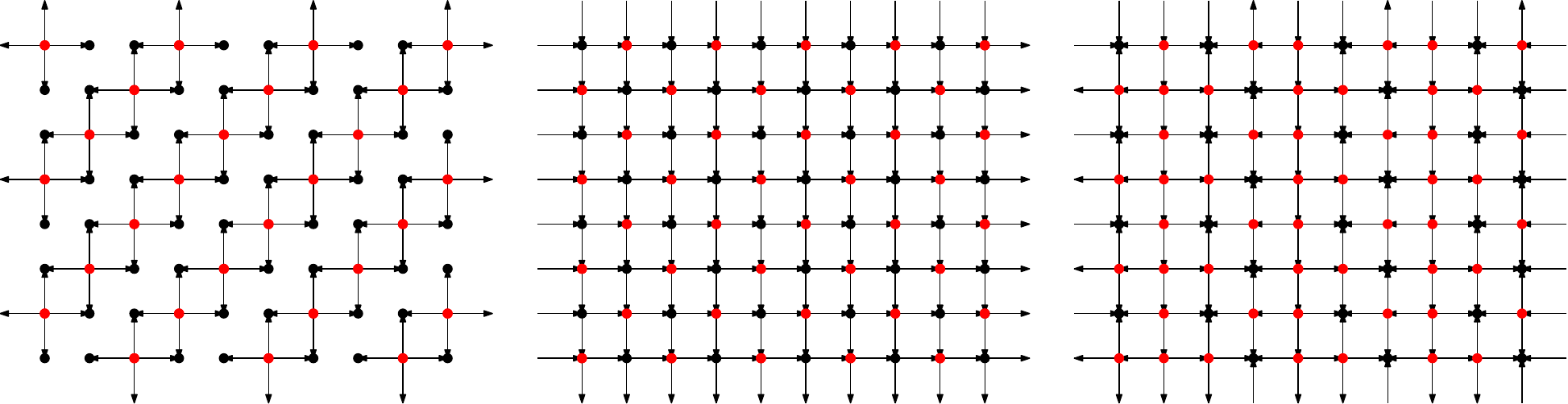}
    \caption{Left to right: orientations of the infinite grid which efficiently achieve $(2,2)$ broadcast vertex density $\frac{1}{3}, \frac12, \frac23$, respectively. In the leftmost graph, some arcs are not shown to denote that these arcs may be oriented arbitrarily without changing the density.} \label{img:infinite,density1/3:2/3}
\end{figure}

Motivated by Theorem~\ref{thm:one_third_two_third} and our prior study of the directed $(t,r)$ broadcast domination interval, we pose the following open problems for further study. 

\begin{que}\label{ques:one_third_two_third}
Does every rational number $d \in \left[\frac13,\frac23\right]$ appear as a $(2,2)$ directed broadcast domination density? If not, classify the rational numbers in that interval that do appear.
\end{que}

\begin{que}\label{ques:ralph}
Is there a $(2,2)$ directed broadcast domination density $d$ such that the density $d$ is irrational, i.e., $ d \in  \left[\frac13,\frac23\right] \cap \QQ^\text{c} ?$ 
\end{que}

One may also ponder the implications of periodic tower placement on various topological surfaces.

\begin{que}\label{question:Adams}
Do the three $(t,r)$ directed broadcast dominating sets given in Figure \ref{img:infinite,density1/3:2/3} result some $(t,r)$ broadcast domination set on a finite toroidal grid? More generally, for any orientation of the grid which results in a rational density of towers, is it possible to realize that orientation and domination density on the toroidal grid?
\end{que}

We provide a partially negative answer to Question~\ref{question:Adams} by noting that the infinite grid achieving tower density~$\frac{1}{3}$ in Figure~\ref{img:infinite,density1/3:2/3} can be oriented to have an aperiodic orientation (by orienting the missing arcs accordingly), but this orientation always has rational tower density. We know that regardless of the orientation of the missing arcs, the resulting $(t,r)$ broadcast dominating set remains efficient because those missing arcs carry no signal. Of course, this infinite grid may also be oriented periodically to achieve the same effect, and we note that one reason for this flexibility in orienting the arcs is that no signal travels along these arcs. In other words, for each of these arcs $(u,v)$ left blank in Figure~\ref{img:infinite,density1/3:2/3}, the signal at the tail $u$ is only 1, so the signal at $v$ coming from $u$ is~0. 

Still, it is the case that all three of these domination densities are realizable as dominating sets on the same toroidal grid, and it is important to note that this fact might help us determine which rational numbers are possible dominating densities. Exploring this further may provide connections to topology and also provide an answer to Question~\ref{ques:one_third_two_third}.

Interestingly, the orientations of the infinite grid depicted in Figure~\ref{img:infinite,density1/3:2/3} also grant a sub-interval of $\dbdi{2,2}{G_{m,n}}$, as we detail below.

\begin{prop}
Given integers $m, n$,
$$
\dbdi{2,2}{G_{m,n}} \supseteq \begin{cases} 
\left[\floor{\frac{mn}{3}}, \floor{\frac{2mn}{3}}\right] &\mbox{if } n \mod{3} = 0\\
\left[\floor{\frac{mn}{3}}, \floor{\frac{2m(n-1)}{3}}\right] &\mbox{if } n \mod{3} = 1\\
\left[\floor{\frac{mn}{3}}, \floor{\frac{4mn + 5m}{6}}\right] &\mbox{if } n \mod{3} = 2\\
\end{cases}.
$$
\end{prop}

\begin{proof}
To achieve each of these values, align the leftmost column of $G_{m,n}$ with a column of the infinite graph of density 2/3 which contains only towers, as in Figure~\ref{fig:fittingGridIntoInfinite}. Now, using this starting column, bound a rectangular region of dimensions $m$ by $n$ to create the directed graph $\vec{G}_{m,n}$. 
We first claim that the set of towers within this region on the infinite graph also forms a dominating set of $\vec{G}_{m,n}$. Furthermore, we claim that this set achieves the domination number of $\vec{G}_{m,n}$. 

To establish the first claim, notice that all vertices within this region with in-degree 0 or 1 are dominating vertices. All other vertices, which have in-degree at least 2 and not greater than 4, are non-dominating vertices. This implies that the set of dominating vertices in $\vec{G}_{m,n}$ forms a dominating set of the graph. 
The second claim is established by the fact that every vertex with in-degree 0 or 1 must be a dominating vertex, as it cannot receive sufficient reception from only a single other vertex. 
Additionally, any other vertex (any vertex with in-degree greater than 1) is already not a dominating vertex. Thus, this set of dominating vertices is minimal.
\end{proof}

\begin{figure}
    \centering
    \includegraphics[width = 0.35\textwidth]{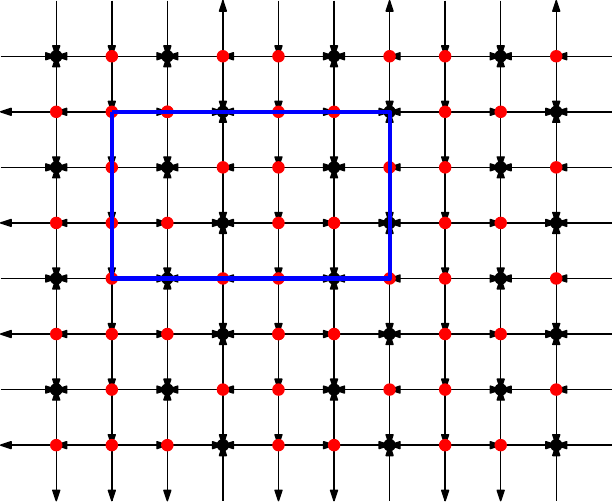}
    \caption{A region in which the leftmost column of a $4 \times 6$ grid graph is aligned with a column of the infinite grid containing only towers.}
    \label{fig:fittingGridIntoInfinite}
\end{figure}

It is also important to note that, much like in the case of finite graphs, there exist infinite grid graphs that cannot be oriented to preserve their $(t,r)$ broadcast dominating sets. One example is on the infinite triangular grid, which cannot be oriented to preserve the $(4,3)$ dominating set given by Harris et al.\ in \cite{harris2018broadcast}.
\begin{figure}[h!]
    \centering
    \includegraphics[width = 0.35\textwidth]{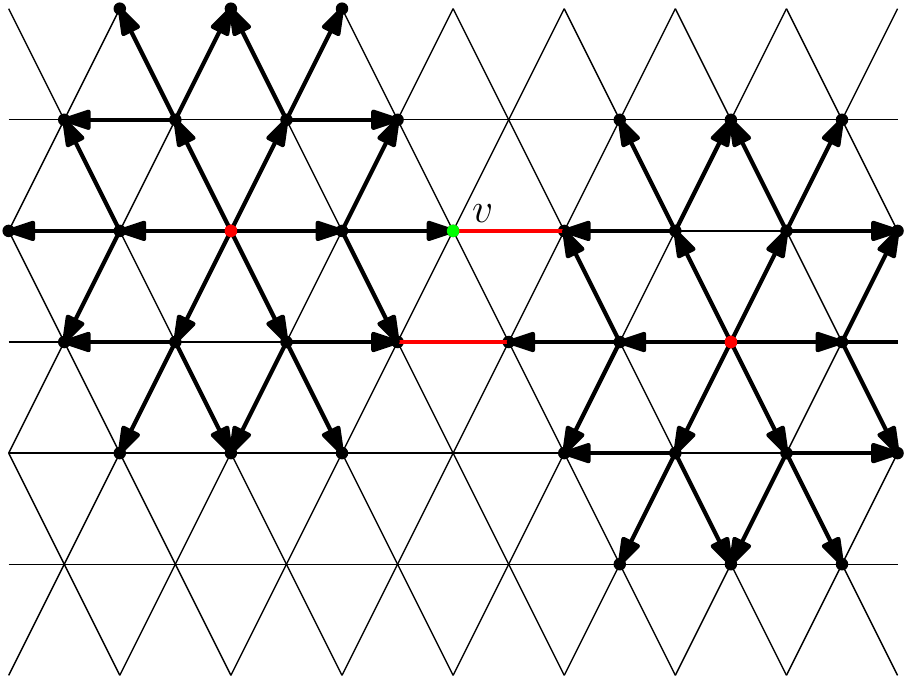}
    \caption{It is impossible to orient the triangular lattice and achieve the efficient $(4,3)$ dominating set. The problematic arcs are highlighted in red.}
    \label{img:(4,3)onTriLattice}
\end{figure}
\begin{ex} 
As pictured in Figure \ref{img:(4,3)onTriLattice}, it is impossible to orient the triangular lattice to preserve the dominating set given by Harris et al.\ in \cite{harris2018broadcast}. In order for the distance-3 out-neighborhood of each broadcasting vertex to be preserved, the black arcs in this figure must be oriented away from the broadcasting vertices. This figure shows just one possible orientation, as some of these black arcs may be reversed from their pictured orientations while preserving the distance-3 out-neighborhood of the dominating vertices. However, under no orientation of the red edges can signal travel ``both ways,'' which would be required for the set to be a $(4,3)$ dominating set. As pictured in Figure~\ref{img:(4,3)onTriLattice}, the reception at vertex $v$ would be $\vec{r}(v) = 3$ if the red edge incident to it were doubly-directed, but orienting that edge only allows for $\vec{r}(v) \leq 2$.
\end{ex}

\section{Future Directions}
\label{chap:futureDirections}

In Section~\ref{chap:broadcast_domination_interval}, we prove that the directed $(t,r)$ broadcast domination interval is full when $t \geq r = 1$ and in the isolated case $(t,r) = (2,2)$. We conjecture that this is true for arbitrary $t \geq r \geq 1$. One direction of study is proving or providing a counterexample to the following.

\begin{conj}
Let $t$ and $r$ be positive integers such that $t \geq r$. Given a graph $G$, let $\dbdi{t,r}{G} = [d,D]$, where 
\begin{align*}
& d=\min\{\gamma_{t,r}(\vec{G}):\vec{G} \text{ is an orientation of } G\} \text{ and } \\ 
& D=\max\{\gamma_{t,r}(\vec{G}): \vec{G} \text{ is an orientation of } G\}. 
\end{align*}
For every $b \in \Z \cap [d,D]$, there exists an orientation $\vec{G}_b$ with $\gamma_{t,r}(\vec{G}_b) = b$.
\end{conj}

We also study directed $(t,r)$ broadcast domination on small grid graphs. There are many possible directions for further study, chief of which would be finding a closed formula for the directed $(t,r)$ broadcast domination interval of an arbitrary finite grid graph.

\begin{proj}
Let $t,r,m,n$ be positive integers, and let $t \geq r \geq 1$. Find a formula for $\dbdi{t,r}{G_{m,n}}$ as a function of $t,r,m,$ and $n$.
\end{proj}

Lastly in Section~\ref{chap:broadcast_domination_interval}, we study the directed $(t,r)$ broadcast domination interval for the star graph. In light of those results, one could study the directed $(t,r)$ broadcast domination interval of other graph families.

\begin{proj}
Let $t$ and $r$ be positive integers such that $t \geq r \geq 1$. For a family of graphs, find the $\interval{t,r}$ of that graph family as a function of $t$, $r$, and the number of vertices in the graph. Examples of interesting graphs to consider include cycles, tournaments (directed complete graphs), and spiders (a collection of paths which all share one endpoint).
\end{proj}

In Section~\ref{chap:infinite_grid} we introduce directed $(t,r)$ broadcast domination on the infinite grid, and we show an orientation of the infinite grid which efficiently achieves broadcasting vertex densities $\frac13$, $\frac12$, and $\frac{2}{3}$ when $(t,r) = (2,2)$. One future direction of study concerns answering the following question.

\begin{que}
Let $(t,r) = (2,2)$. For $d \in \Q \cap \left[\frac{1}{3}, \frac{2}{3}\right]$, does there exist an orientation of the infinite Cartesian grid which achieves tower density $d$? If not, for which rational numbers $d\in[\frac13,\frac23]$ can densities be achieved?
\end{que}

There are many other ways to study directed $(t,r)$ broadcast domination on the infinite grid. One potential direction could be to explore other known efficient $(t,r)$ broadcast dominating patterns (i.e., those studied in \cite{blessing2014tr}), asking whether there exist orientations of their respective infinite graphs which still allow that dominating pattern to form a directed $(t,r)$ broadcast dominating set. Another avenue for further study is to investigate other values of $t$ and $r$ not previously studied.

\section*{Acknowledgements}
P.~E.~Harris acknowledges funding support from  The Karen EDGE Fellowship Program and P.~Hollander thanks Williams College Science Center for funding in support of this research.

\bibliographystyle{plainnat}
\bibliography{Bibliography}

\end{document}